\documentclass[12pt]{amsart}

\usepackage{amssymb,latexsym,amsmath,mathtools,nicefrac,bbm}
\usepackage{stmaryrd}
\usepackage{graphicx}
\usepackage{psfrag}
\usepackage{color}
\usepackage{wrapfig}
\usepackage{tikz}
\usepackage{enumitem}
\setlist[enumerate]{parsep=0pt plus 4pt,topsep=0pt plus 4pt}
\usepackage{url}
\usepackage{hyperref}
\allowdisplaybreaks

\newcommand{\excise}[1]{}

\newtheorem{thm}{Theorem}[section]
\newtheorem{lemma}[thm]{Lemma}

\newtheorem{cor}[thm]{Corollary}
\newtheorem{prop}[thm]{Proposition}

\newtheorem{prob}[thm]{Problem}

\theoremstyle{definition}
\newtheorem{example}[thm]{Example}
\newtheorem{remark}[thm]{Remark}
\newtheorem{defn}[thm]{Definition}
\newtheorem{conv}[thm]{Convention}

\numberwithin{equation}{section}

\newcounter{separated-sec}

\newcounter{sylvan-sec}
\newcounter{sylvan}

\newcommand{\Ring}[1]{\ensuremath{\mathbb{#1}}}

\renewcommand\>{\rangle}

\newcommand\1{\mathbbm{1}}
\newcommand\HH{{\widetilde H}{}{}}
\newcommand\KK{\mathbb{K}}
\newcommand\NN{\Ring{N}}
\newcommand\PP{\Ring{P}}
\newcommand\QQ{\Ring{Q}}
\newcommand\RR{\Ring{R}}
\newcommand\CC{\Ring{C}}

\newcommand\ZZ{\Ring{Z}}
\newcommand\bb{{\mathbf b}}
\newcommand\cc{{\mathbf c}}
\newcommand\ee{{\mathbf e}}
\newcommand\kk{\Bbbk}

\newcommand\xx{{\mathbf x}}
\newcommand\yy{{\mathbf y}}
\newcommand\zz{{\mathbf z}}
\newcommand\del{\partial}
\renewcommand\aa{{\mathbf a}}
\renewcommand\phi{\varphi}

\newcommand\bs{\backslash}

\newcommand\oS{{\hspace{.2ex}\ol{\hspace{-.2ex}S}}}
\newcommand\oT{{\ol T}}
\newcommand\oU{{\hspace{.2ex}\ol{\hspace{-.2ex}U\hspace{-.05ex}}\hspace{.05ex}}}

\newcommand\from{\leftarrow}
\newcommand\into{\hookrightarrow}
\newcommand\otni{\hookleftarrow}
\newcommand\onto{\twoheadrightarrow}
\newcommand\otno{\twoheadleftarrow}

\newcommand\spot{{\hbox{\raisebox{1pt}{\tiny$\scriptscriptstyle\bullet$}}}}
\newcommand\ffrom{\longleftarrow}
\newcommand\minus{\smallsetminus}
\newcommand\simto{\mathrel{\!\ooalign{$\fillrightmap$\cr\raisebox{.75ex}{$\,\sim\ \hspace{.2ex}$}}}}
\newcommand\sspot{{\spot\spot}}

\newcommand\verteq{\rotatebox{90}{\,$=\,$\,}}
\newcommand\nothing{\varnothing}

\newcommand\subspot{{\begin{array}{@{}c@{}}\\[-6.1ex]
		     \hspace{-.1ex}\hbox{\Large$\cdot$}
		     \\[-4ex]\end{array}\!}}

\newcommand\downmapsto{\rotatebox{-90}{$\mapsto\ $}}
\newcommand\fillbar{\mkern-6mu\cleaders\hbox{$\mkern-2mu \mathord- \mkern-2mu$}\hfill
	\mkern-6mu \mathord-}
\newcommand\filleftmap{\mathord\leftarrow \mkern-6mu
	\cleaders\hbox{$\mkern-2mu \mathord- \mkern-2mu$}\hfill
	\mkern-6mu \mathord-}
\newcommand\fillrightmap{\mathord- \mkern-6mu
	\cleaders\hbox{$\mkern-2mu \mathord- \mkern-2mu$}\hfill
	\mkern-6mu \mathord\rightarrow}

\newcommand\fillotno{\mathord\twoheadleftarrow \mkern-6mu
	\cleaders\hbox{$\mkern-2mu \mathord- \mkern-2mu$}\hfill
	\mkern-6mu \mathord-}
\newcommand\longuparrow{\left\uparrow\rlap{\phantom{$\sum_i^j$}}\right.}
\newcommand\longdownarrow{\left.\llap{$\phantom{\sum_i^j}$}\right\downarrow}

\renewcommand\epsilon{\varepsilon}
\renewcommand\implies{\Rightarrow}

\newcommand\st{\mathit{ST\hspace{-.2ex}}}

\newcommand\nf[2]{{\nicefrac{#1\!}{\!#2}}}
\newcommand\ol[1]{{\overline{#1}}}

\newcommand\wt[1]{{\widetilde{#1}}}
\newcommand\mkl[1]{\makebox[0pt][l]{$#1$}}
\newcommand\mko[1]{\makebox[0pt][c]{$#1$}}
\newcommand\nfn[2]{\nf{{\sst-\!}#1}{#2}}
\newcommand\betti[3]{{\beta_{{#1},{#2}}(#3)}}
\newcommand\edgeup[1]{/\makebox[0pt][l]{\raisebox{-.4ex}{$\!\scriptstyle#1$}}}
\newcommand\edgedown[1]{\bs\makebox[0pt][l]{\raisebox{.7ex}{$\!\scriptstyle#1$}}}
\newcommand\edgehoriz[1]{\raisebox{0pt}[0pt][0pt]{$\stackrel{#1}{\hbox{ --- }}$}}

\DeclareMathOperator\gr{gr} 
\DeclareMathOperator\tor{Tor} 

\DeclareMathOperator\image{im} 

\newcommand\dis{\displaystyle}
\newcommand\sst{\scriptscriptstyle}

\newcommand\foot{\footnotesize}

\newcommand\monomialmatrix[3]{{
\begin{array}{@{}r@{\:}r@{}c@{}l@{}}
  \begin{array}{@{}c@{}}		
	\begin{array}{@{}r@{}}
	\\
	#1
	\end{array}\!
  \end{array}						
&
  \begin{array}{@{}c@{}}		
	\begin{array}{@{}l@{}}\\				
	\end{array}						
	\\							
	\left[\begin{array}{@{}l@{}}				
	#3							
	\end{array}\!						
	\right.							
  \end{array}							
&
  #2					
&
  \begin{array}{@{}c@{}}		
	\begin{array}{@{}l@{}}\\				
	\end{array}						
	\\							
	\left.\!\begin{array}{@{}l@{}}				
	#3							
	\end{array}						
	\right]							
  \end{array}							
\end{array}
}}

\newcommand\red[1]{\color{red}#1}
\newcommand\blu[1]{\color{blue}#1}

\begin{document}

\mbox{}
\vspace{-4.1ex}
\title{Minimal resolutions of monomial ideals}
\author{\vspace{-1.3ex}John Eagon}
\address{School of Mathematics\\University of Minnesota\\Minneapolis, MN 55455}
\author{Ezra Miller}
\address{Mathematics Department\\Duke University\\Durham, NC 27708}
\urladdr{\url{http://math.duke.edu/people/ezra-miller}}
\author{Erika Ordog}
\address{Mathematics Department\\Duke University\\Durham, NC 27708}
\urladdr{\url{https://fds.duke.edu/db/aas/math/grad/ordog}}

\makeatletter
  \@namedef{subjclassname@2010}{\textup{2010} Mathematics Subject Classification}
\makeatother

\date{15 May 2020}

\begin{abstract}
An explicit, closed-form combinatorial minimal free resolution of an
arbitrary monomial ideal~$I$ in a polynomial ring in~$n$ variables
over a field of characteristic~$0$ (and almost all positive
characteristics) is defined canonically, without any choices, using
higher-dimensional generalizations of combined spanning trees for
cycles and cocycles (\emph{hedges}) in the upper Koszul simplicial
complexes of~$I$ at lattice points in~$\ZZ^n$.  The differentials in
these \emph{sylvan resolutions} are expressed as matrices whose
entries are sums over lattice paths of weights determined
combinatorially by sequences of hedges (\emph{hedgerows}) along each
lattice path.  This combinatorics enters via an explicit
matroidal expression for Moore--Penrose pseudoinverses as weighted
averages of splittings defined by hedges.  The translation from
Moore--Penrose combinatorics to free resolutions relies on Wall
complexes, which construct minimal free resolutions of graded ideals
from vertical splittings of Koszul bicomplexes.  The algebra of Wall
complexes applied to individual hedgerows yields explicit but
noncanonical combinatorial minimal free resolutions of arbitrary
monomial ideals~in~any~\mbox{characteristic}.
\end{abstract}
\maketitle

\tableofcontents

\section{Introduction}\label{s:intro}

\vspace{-1ex}
\subsection*{Overview}\label{sub:overview}
\mbox{}

\noindent
Irving Kaplansky had a habit of circulating to his students precise
problem lists for potential dissertation topics.  One such list
contained a problem on ideals generated by subdeterminants of a
matrix, which resulted in \cite{eagon-northcott1962}.  A later list,
we speculate, concerned monomial ideals, resulting in Taylor's thesis
\cite{taylor1966}, the first general construction of free resolutions
for arbitrary monomial ideals.  Since then the problem of finding
minimal free resolutions of monomial ideals in polynomial rings has
been central to the combinatorial side of commutative algebra,
stimulating an enormous amount of research on the algebraic,
combinatorial, and homological structure of monomial ideals, including
hundreds of research papers and several influential books.  The
ultimate goal is a free resolution that is universal, canonical, and
minimal, with closed-form combinatorial formulas for the
differentials.  This means that the construction should work for any
monomial ideal, involve no choices, have no redundancy in algebraic or
numerical senses, and be explicit in terms of the discrete input
that determines a monomial ideal.


The sylvan resolutions introduced here use Moore--Penrose
pseudoinverses of differentials, combinatorially characterized in
terms of \mbox{higher-dimensional} analogues of spanning trees for
cycles and cocycles following Berg \cite{berg1986}, to produce a
universal, canonical, closed-form combinatorial construction of
minimal free resolutions of arbitrary monomial ideals over fields of
characteristic~$0$ and most positive characteristics.  In any
characteristic, the spanning-tree framework produces noncanonical but
nonetheless universal combinatorial constructions of minimal free
resolutions of monomial ideals.

\enlargethispage*{1.5ex}
\vspace{-.5ex}
\subsubsection*{Acknowledgements}
A great debt goes to Joel Roberts, with whom we discussed preliminary
ideas along these lines decades ago.  EM wishes to thank Art Duval for
an enlightening conversation about simplicial spanning trees.
Alexandre Tchernev provided extremely valuable feedback on a preprint
version of this paper.
EM and EO had support from NSF DMS-1702395.
EO was supported for a semester by NSF~\mbox{DMS-1406371}.

\subsection{Prior work}\label{sub:history}
\mbox{}

\noindent
Formulas for the Betti numbers---the ranks of the free modules in a
minimal free resolution---based on the combinatorial topology of
simplicial complexes have been known since the 1970s through work of
Hochster \cite{hochster1977} and others,
but the differentials of the resolutions have remained elusive.  All
prior studies that produce differentials in resolutions of monomial
ideals have dispensed with one or more of the desired properties.  For
example, Taylor's resolution \cite{taylor1966} is not minimal;
Lyubeznik's improvement on it \cite{lyubeznik1988} is not minimal or
canonical; Eliahou--Kervaire resolutions of stable ideals
\cite{eliahou-kervaire1990} are not universal; Eagon's resolutions by
Wall complex \cite{eagon1990} are not a~priori combinatorial or
canonical, although the point of our work here is that combinatorics
of Moore--Penrose pseudoinverses remedies both of these; hull
resolutions \cite{bayer-sturmfels1998} are not minimal; Scarf
resolutions of generic monomial ideals
\cite{bayer-peeva-sturmfels1998,miller-sturmfels-yanagawa2000} are not
universal, although they can be made universal by generic deformation,
sacrificing minimality and canonicality; Yuzvinsky's resolutions using
splittings in the manner of Eagon, \mbox{applied} to LCM lattices
instead of Koszul simplicial complexes \cite{yuzvinsky1999}, are not
a~priori combinatorial, and Yuzvinsky claims they are not canonical,
although Moore--Penrose pseudoinverses remedy the latter and can in
principle remedy the former (a topic for future work); resolutions for
shellable monomial ideals \cite{batzies-welker2002} are not universal;
planar graph resolutions \cite{miller-Planar2002} are not universal,
being defined only in three variables and not even canonical there;
resolutions supported on order complexes of Betti posets
\cite{tchernev-varisco2015} are not minimal; and Buchberger
resolutions~\cite{olteanu-welker2016}~are~not~canonical~or~minimal.

The strongest combinatorial structural result for minimal free
resolutions of monomial ideals available before the current work is
that they all admit hcw-poset structures
\cite{clark-tchernev2019}---essentially poset generalizations of the
cellular structures in \cite{bayer-sturmfels1998}.  However, this
poset framework does not construct resolutions but rather imposes
structures a~posteriori on a given resolution.  That said, in a
development subsequent to the current work, \mbox{Tchernev} produced
canonical, minimal, universal resolutions that are not closed-form but
can be seen as combinatorial in the sense of being algorithmic
\cite{tchernev2019}.

\subsection{Koszul simplicial complexes and Hochster's formula}\label{sub:koszul}
\mbox{}

\noindent
The combinatorics of minimal resolutions of monomial ideals is
grounded in the local combinatorics near lattice points in the
partially ordered set of exponent vectors.

\begin{defn}\label{d:Kb}
For a monomial ideal~$I$ and a nonnegative integer vector $\bb \in
\NN^n$ with~$n$ entries, the \emph{(upper) Koszul simplicial complex}
of~$I$ in degree~$\bb$ is
$$%
  K^\bb I
  =
  \{\tau \in \{0,1\}^n \mid \xx^{\bb-\tau}\in I\}.
$$
\end{defn}

That is, standing at the lattice point~$\bb$, thought of as an
exponent vector on a monomial in the ideal~$I$, one looks backward to
see which (combinations of distinct) coordinate directions one can
move along to remain in~$I$.

\begin{thm}[Hochster's formula]\label{t:Kb}
Fix a monomial ideal $I \subseteq \kk[\xx]$ and a degree vector $\bb
\in \NN^n$.  There is a natural isomorphism of vector spaces
\begin{align*}
  \tor_i(\kk,I)_\bb
  &=
  \HH_{i-1} (K^\bb I;\kk).
\\
\intertext{Consequently, the Betti numbers of~$I$ in degree~$\bb$ can
be expressed as}
  \betti i\bb I
  &=
  \dim_\kk \HH_{i-1}(K^\bb I;\kk).
\end{align*}
\end{thm}

For an exposition and proof, see \cite[Theorem~1.34]{cca} and
surrounding material.

Theorem~\ref{t:Kb} is the sense in which simplicial homology
categorifies monomial Betti numbers.  At issue in Kaplansky's problem
is how to categorify the differentials in a minimal free resolution
of~$I$.  More precisely, any attempt to produce general minimal free
resolutions of arbitrary monomial ideals must reduce---explicitly or
implicitly---to solving the following concrete problem.

\begin{prob}\label{prob:kaplansky}
For a monomial ideal $I \subseteq \kk[\xx]$, produce vector space
homomorphisms
\vspace{-.5ex}
\begin{align*}
  \bigoplus_{\aa \prec \bb}\HH_{i-1} (K^\aa I;\kk)
  &\from
  \HH_i(K^\bb I;\kk)
\vspace{-.6ex}
\end{align*}
for all $i \!\in\! \NN$ and multigraded degrees $\bb \!\in\! \NN^n$
whose induced\/ $\kk[\xx]$-module \mbox{homomorphisms}%
\vspace{-.5ex}
\begin{align*}
  \bigoplus_{\aa\in\NN^n}\HH_{i-1}(K^\aa I;\kk)\otimes_\kk\kk[\xx](-\aa)
  &\from
  \bigoplus_{\bb\in\NN^n}\HH_i(K^\bb I;\kk)\otimes_\kk\kk[\xx](-\bb)
\vspace{-.6ex}
\end{align*}
constitute a free resolution of~$I$.
\end{prob}

Such a free resolution would automatically be minimal, by the Betti
number computation in Theorem~\ref{t:Kb}.  To connect the vector space
homomorphisms to the induced module homomorphisms in more detail,
first see the right-hand side of the vector space homomorphism as
$\tor_{i+1}(\kk,I)_\bb$.  Thinking of it as (the $\kk$-linear span of)
a basis for the $(i+1)^\mathrm{st}$~syzygies in degree~$\bb$, the
differential in a minimal free resolution preserves the degree~$\bb$
while taking this basis to homological stage~$i$.  The summands in
homological stage $i$ that contribute nonzero components to
degree~$\bb$ have the natural form \mbox{$\tor_i(\kk,I)_\aa
\otimes_\kk \kk[\xx]$} for some $\aa \prec \bb$.  The degree~$\bb$
component of this free summand~is
$$%
  \bigl(\tor_i(\kk,I)_\aa \otimes_\kk \kk[\xx]\bigr){}_\bb
  =
  \xx^{\bb-\aa}\tor_i(\kk,I)_\aa
  =
  \HH_{i-1}(K^\aa I;\kk)
$$
as an ungraded vector space.  (To keep track of the grading, the
left-hand side vector space here would have to be shifted into
multigraded degree~$\bb$.)  Taking the direct sum over $\aa \prec \bb$
yields the left-hand side of the vector space homomorphism in
Problem~\ref{prob:kaplansky}.

\subsection{Sylvan combinatorics of the canonical differential}\label{sub:combin}
\mbox{}\enlargethispage*{2ex}

\noindent
Given a free resolution whose syzygy modules have specified bases, the
differentials can be expressed by matrices of scalars using monomial
matrices \cite[Section~3]{alexdual}.  However, one of the fundamental
obstacles to overcome in expressing an explicit, closed-form
description of a minimal free resolution of an arbitrary monomial
ideal is how to present a homomorphism canonically between homology
vector spaces, which do not possess natural bases.  Our approach is to
specify a linear map \smash{$\wt C_{i-1} K^\aa I \from \wt C_i K^\bb
I$} from $i$-chains to $(i-1)$-chains using their natural bases but
then ensure that this linear map induces a well defined homology
homomorphism \smash{$\HH_{i-1} K^\aa I \from \HH_i K^\bb I$}.  (The
field~$\kk$ is fixed throughout and suppressed from the notation.)
Thus, given any cycle of dimension~$i$, expressed in the basis of
$i$-simplices in~$K^\bb I$, the closed-form description acts on each
term in the cycle to produce a cycle expressed in the basis of
\mbox{$(i-1)$-simplices} in~$K^\aa I$.  Different input cycles can
yield different output cycles, a~priori, even when they represent the
same homology class, as long as homologous input cycles yield
homologous output cycles.  That said, our formulation takes homologous
cycles to the same cycle, inducing a homomorphism \smash{$\wt Z_{i-1}
K^\aa I \from \HH_i K^\bb I$}.  This is part of the central~result,
Definition~\ref{d:sylvan-matrix} and Theorem~\ref{t:sylvan}: the
closed-form specification of the differential in our \emph{canonical
sylvan resolution} in characteristic~$0$ and most positive
characteristics.

In more detail, Definition~\ref{d:sylvan-matrix} and
Theorem~\ref{t:sylvan} formulate an explicit linear map
$$%
  \smash{\wt C_{i-1} K^\aa I \stackrel{\ D}\ffrom \wt C_i K^\bb I}.
$$
Specifying~$D$ is the same as specifying its entries~$D_{\sigma\tau}$
for $\tau \in \wt C_i K^\bb I$ and $\sigma \in \wt C_{i-1} K^\aa I$.
(See Convention~\ref{conv:block-matrix} for the resulting matrix
notation.)  The combinatorics is matroidal, generalizing that of
spanning trees in graphs, applied to the upper Koszul simplicial
complexes of~$I$ at lattice points in~$\ZZ^n$.  The entries
$D_{\sigma\tau}$ are expressed as weighted sums over all saturated
decreasing lattice paths from $\bb$ to~$\aa$, where the weights come
from
\begin{itemize}
\item%
coefficients of faces in unique circuits or boundaries obtained by
throwing one additional facet into a higher-dimensional analogue of a
spanning tree, and
\item%
determinants of submatrices indexed by the appropriate rows and
columns.
\end{itemize}
The determinants are unavoidable in high dimension.  They reflect the
fact that~\mbox{integer} boundaries of individual faces of can
contribute to bases for sublattices of varying index.

\subsection{Methods}\label{sub:methods}
\mbox{}

\noindent
The apparent obstruction to constructing closed-form minimal
resolutions has been how to appropriately relate the various Koszul
simplicial homology groups that categorify the multigraded Betti
numbers.  Canonical homomorphisms among subquotients of the relevant
homology groups constitute the spectral sequence
(Corollary~\ref{c:KK-spectral-seq}) of an appropriately constructed
Koszul bicomplex (Definition~\ref{d:KK}), but alas, there is no
categorically natural way to lift these homomorphisms on subquotients
to homomorphisms on the intact homology.  The method of Wall complexes
here, following Eagon \cite{eagon1990} (see Section~\ref{s:wall} for
an exposition), observes that any choice of splitting for the vertical
differential forces the Koszul simplicial homology to split in such a
way that the spectral sequence differentials collate into a compendium
differential that solves Problem~\ref{prob:kaplansky}.

The splittings, and subsequently the Wall differentials, are made
canonical by using Moore--Penrose pseudoinverss, which explains the
exclusion of finitely many positive characteristics.  The splittings
are then made combinatorially explicit by applying summation formulas
for the pseudoinverse \cite[Theorem~1]{berg1986} and
\cite[Theorem~2.1]{benTal-teboulle1990}, which are so rarely cited
that they must be largely unknown to algebraists.  We rephrase these
formulas as weighted averages of splittings (Corollary~\ref{c:hedge}
and Proposition~\ref{p:hedge}) in the context of the theory of
higher-dimensional analogues of spanning trees initiated by Kalai
\cite{kalai1983}, as developed by Duval, Klivans, and Martin
\cite{duval-klivans-martin2009, duval-klivans-martin2011}, Petersson
\cite{petersson2009}, and Lyons \cite{Lyons2009}.  Versions of
Corollaries~\ref{c:hedge}
and~\ref{c:projection-hedge} in which the determinants are interpreted
as orders of certain torsion subgroups in homology were proved by
Catanzaro, Chernyak, and Klein
\cite{catanzaro-chernyak-klein2015,catanzaro-chernyak-klein2017}.

\subsection{Noncanonical sylvan resolutions}\label{sub:noncanonical}
\mbox{}

\noindent
Other combinatorial splittings of the vertical Koszul differential,
arising from individual hedgerows
(Definition~\ref{d:hedgerow} with Definition~\ref{d:hedge} and
Example~\ref{e:CW}) contributing summands to the Moore--Penrose
pseudoinverse formula, produce perfectly good combinatorial minimal
resolutions.  These could be suited to algorithmic computation
(Remark~\hspace{-.2ex}\ref{r:computation}).  Moreover, these
\mbox{splittings---and} hence the corresponding sylvan minimal
resolutions---require no division and hence are defined over any field
(Corollary~\ref{c:noncanonical-sylvan}).  Certain existing families of
minimal resolutions appear to be sylvan, with apt choices of hedges;
that is, they can be constructed as Wall complexes for
\mbox{suitable}~\mbox{splittings}~(\mbox{Remark}~\ref{r:eliahou-kervaire}).

\subsection{Logical structure}\label{sub:logical}
\mbox{}

\noindent
To make the prerequisites clear, the paper begins with a short, direct
path to a rigorous statement of the main result---the combinatorial
description of canonical minimal free resolutions of monomial ideals
in Theorem~\ref{t:sylvan}.  Thus Sections~\ref{s:hedge}
and~\ref{s:sylvan} are self-contained introductions to the relevant
simplicial notions and the combinatorial assembly of these along
descending lattice paths.  The proof of Theorem~\ref{t:sylvan} must
wait until Section~\ref{s:resolutions}, as it relies on the Hedge
Formula (Corollary~\ref{c:hedge}), the Wall construction of minimal
free resolutions via Koszul splittings
(Corollary~\ref{c:wall-koszul}), and the Koszul simplicial formula for
those (Theorem~\ref{t:choice-of-splitting}).  No intervening result
relies on the statement of Theorem~\ref{t:sylvan}.

\subsection{Conventions}\label{sub:conv}
\mbox{}\vspace{-.5ex}

\begin{conv}[Cellular notions]\label{conv:simplicial}
A CW complex~$K$ has its set $K_i$ of $i$-faces and integer reduced
chain groups $\wt C_i^\ZZ K = \ZZ\{K_i\}$ with differential~$\del_i:
\wt C_i^\ZZ K \to \wt C_{i-1}^\ZZ K$.  Tensoring with any field~$\kk$,
such as the fields~$\QQ$ or~$\CC$ of rational or complex numbers,
yields the reduced chain complex $\wt C_\spot^\kk K = \kk \otimes_\ZZ
\wt C_\spot K = \kk\{K_i\}$ over~$\kk$, with differential also
denoted~by~$\del$.  The same conventions hold for cycles $\wt Z_i K =
\ker\del_i \subseteq\nolinebreak \wt C_i K$,
\mbox{boundaries}~\mbox{$\wt B_i K = \image\del_{i+1} \subseteq \wt
C_i K$}, and reduced homology $\HH_i K =\nolinebreak \wt Z_i K / \wt
B_i K$.  To unclutter the notation, it helps to omit the superscript
$\ZZ$ or~$\kk$ when the context is clear.  The sign on a facet
$\sigma$ of a cell~$\tau$ in the boundary~$\del\tau$ is written
$(-1)^{\sigma \subset \tau}$.
\end{conv}

\begin{conv}[Polynomial and monomial notions]\label{conv:monomial}
Fix, once and for all, an ideal~$I$ in the polynomial ring $\kk[\xx]$
in $n$ variables $\xx = x_1,\dots,x_n$ over a field~$\kk$ that is
assumed throughout to be arbitrary unless otherwise stated.  Assume
that $I$ is a monomial ideal unless otherwise explicitly stated.
Monomials in~$\kk[\xx]$ are denoted by $\xx^\aa$ for lattice points
$\aa \in \NN^n$.  Unadorned tensor products $\otimes$ are understood
as~$\otimes_\kk$.
\end{conv}

\section{Shrubs, stakes, and hedges}\label{s:hedge}

\begin{defn}\label{d:hedge}
Fix nonegative integers $m,n \in \NN$.
\begin{enumerate}\itemsep=.3ex
\item\label{i:shrubbery'}%
A \emph{shrubbery} for a surjection $B \otno \kk^n$ is a subset $T
\subseteq \{e_1,\dots,e_n\}$ of the standard basis such that the
composite $B \otno \kk^n \otni \kk\{T\}$ is an isomorphism~$\del_T$.

\item\label{i:stake'}%
A \emph{stake set} for an injection $\kk^m \otni B$ is a subset $S
\subseteq \{e_1,\dots,e_m\}$ of the standard basis such that the
composite $\kk\{S\} \otno \kk^m \otni B$ is an isomorphism~$\del_S$.
Basis vectors in a stake set are called \emph{stakes}.

\item%
If \smash{$\kk^m \stackrel\del\ffrom \kk^n$} is a linear map with
image $B = \del(\kk^n)$, then a \emph{hedge} for $\del$ is a
choice~$\st$ of a \emph{shrubbery~$T$ for~$\del$} as in
item~\ref{i:shrubbery'} and a \emph{stake set~$S$ for~$\del$}
as~in~item~\ref{i:stake'}.%
\end{enumerate}\enlargethispage*{2.5ex}
\end{defn}

\begin{example}\label{e:CW}
Fix a CW complex~$K$ and a field~$\kk$.
\begin{enumerate}\itemsep=.3ex
\item\label{i:shrubbery}%
A \emph{shrubbery} in dimension~$i$ is a shrubbery $T_i \subseteq K_i$
for $\wt B_{i-1}^\kk K \otno \wt C_i^\kk K$.

\item\label{i:stake}%
A \emph{stake set} in dimension~$i-1$ is a stake set for
\smash{$\wt C_{i-1}^\kk K \otni \wt B_{i-1}^\kk K$}.

\item\label{i:hedge}%
A \emph{hedge} in~$K$ of dimension~$i$ is a choice of shrubbery
in~$K_i$ and stake set in~$K_{i-1}$.
\end{enumerate}
A hedge of dimension~$i$ may be expressed as $\st_i = (S_{i-1},T_i)$.
\end{example}

\begin{remark}\label{r:spanning-tree}
In the literature shrubberies are often known as ``spanning trees'',
or ``spanning forests'', or some variant; see
\cite{catanzaro-chernyak-klein2015} and
\cite{duval-klivans-martin2015}, for example.  We avoid these terms
because they are inapt in certain ways---subcomplexes whose facets
form shrubberies need not be connected in any appropriate sense (so
should not be called trees) if the ambient CW complex is disconnected,
and there could be forests that span in some appropriate sense
but are nonetheless not spanning forests---and their precise
definitions vary from paper to paper.  But they explain our botanical
terminology as well as our shrubbery symbol ``$\,T\,$'', which
classically stands for~``tree''.  Regardless of terminology or
notation, a hedge is matroidal information, given by subsets of fixed
bases, and is hence combinatorial in nature.
\end{remark}

\begin{remark}\label{r:coboundary}
Definitions~\ref{d:hedge}.\ref{i:shrubbery'} and~\ref{i:stake'} are
plainly dual: $S \subseteq \kk^m$ is a stake set for~$\del$ if and
only if the corresponding dual basis vectors are a shrubbery for the
transpose~$\del^\top$.  For this reason, stake sets have been called
``cotrees'' in persistent homology \cite{edelsbrunner-olsbock2018}.
\end{remark}

The horticultural picture that goes with the terminology extends: each
stake is tied to the tip of a unique~shrub: the chain~$s$ in the next
result (see Definition~\ref{d:chain-linked}.\ref{i:chain}).

\begin{lemma}\label{l:chain-linked}
Fix a hedge $\st$ for $\kk^m \stackrel\del\ffrom \kk^n$ and a stake
$\sigma \in S$.  There is a unique chain $s \in \kk\{T\}$ whose
boundary has coefficient~$1$~on~$\sigma$ and~$0$ on all other
stakes~in~$S$.
\end{lemma}
\begin{proof}
In the notation of Definition~\ref{d:hedge}, $s = (\del_S \circ
\del_T)^{-1}\sigma$, so that $\del_S \circ \del_T s = \sigma$.
\end{proof}

The usual property of a spanning tree in a graph is that every edge of
the graph closes a unique circuit with the tree edges.  The analogous
well known combinatorics of shrubberies is most simply described by
way of a trivial lemma.

\begin{lemma}\label{l:alpha}
Fix a field~$\kk$ and a vector subspace $A \subseteq \kk^\ell$.  For
each subset $U \subseteq \{e_1,\dots,e_\ell\}$ of the standard basis
such that $\kk^\ell = \kk\{U\} \oplus A$ there is a linear map
$\alpha_U:\nolinebreak \kk^\ell \onto A$ that takes $\rho \in
\kk^\ell$ to the unique vector $\rho - r \in A$ such that $r \in
\kk\{U\}$.\qed
\end{lemma}

\begin{example}\label{e:unique-vector}
Lemma~\ref{l:alpha} is applied particularly when
\begin{itemize}
\item%
$U = T$ is a shrubbery for the surjection $\kk^\ell/A \otno \kk^\ell$,
in which case~$\alpha_U$ is called the \emph{circuit projection} and
denoted~$\zeta_T$, or
\item%
$U = \oS$ is the complement of a stake set~$S$ for the injection
$\kk^\ell \otni A$, in which case $\alpha_U$ is called the
\emph{boundary projection} and denoted~$\beta_S$.
\end{itemize}
If $T$ and~$S$ come from \smash{$\kk^m \stackrel\del\ffrom \kk^n$},
then $\zeta_T: \kk^n \onto \ker\del$ and $\beta_S: \kk^m \onto B =
\del(\kk^n)$.

In cellular settings, when $T_i \subseteq K_i$ is a shrubbery of
dimension~$i$, the circuit projection \smash{$\zeta_{T_i}: \wt C_i^\kk
K \onto \wt Z_i^\kk K$} has a combinatorial interpretation: every
$i$-face $\tau \not\in T_i$ lies in a unique \emph{$T_i$-circuit}
\smash{$\zeta_{T_i}(\tau) = \tau - t \in \wt Z_i^\kk K$} that is a
cycle with coefficient~$1$ on $\tau$ in the CW complex with facets
$\{\tau\} \cup T_i$.  (If $\tau \in T_i$ then $t = \tau$, so $\tau - t
= 0$.)

Similarly, when $S_i \subseteq K_i$ is a stake set of dimension~$i$,
the boundary projection \smash{$\beta_{S_i}: \wt C_i^\kk K \!\onto\!
\wt B_i^\kk K$} takes every stake to the unique boundary with
coefficient~$1$ on~$\sigma$.  The combinatorial interpretation in
homological stage $i-1$ views the boundary $\beta_{S_{i-1}}(\sigma)
\in \sigma + \kk\{\ol S_{i-1}\}$ as the \emph{hedge~rim} of~$\sigma$:
the boundary of the shrub~$s$ from Lemma~\ref{l:chain-linked} for any
choice of shrubbery~$T_i$; see
Definition~\ref{d:chain-linked}.\ref{i:boundary} and Lemma~\ref{l:bS}.
\end{example}

Circuits, shrubs, and hedge rims are the core simplicial combinatorial
players.

\begin{defn}\label{d:chain-linked}
Fix a CW complex~$K$ and a field~$\kk$.
\begin{enumerate}
\item\label{i:cycle}%
Fix a shrubbery $T_{i-1}$.  An $(i-1)$-face $\sigma$ is
\emph{cycle-linked} to any $(i-1)$-face~$\sigma' \in\nolinebreak
K_{i-1}$ with nonzero coefficient in the circuit
$\zeta_{T_{i-1}}(\sigma)$ defined in Example~\ref{e:unique-vector}.

\item\label{i:chain}%
Fix a hedge $\st_i$ in~$K$.  A stake $\sigma \in S_{i-1}$ is
\emph{chain-linked} to an $i$-face~$\tau \in K_i$ if $\tau$ has
nonzero coefficient in the \emph{shrub} of~$\sigma$: the chain $s$
from~Lemma~\ref{l:chain-linked}.

\item\label{i:boundary}%
Fix a stake set~$S_i$.  An $i$-face $\rho \in K_i$ is
\emph{boundary-linked} to $\rho' \in K_i$ if $\rho'$ has nonzero
coefficient in the \emph{hedge rim} of~$\rho$: the chain $r(\rho) =
\rho - \beta_{S_i}(\rho)$ from~Example~\ref{e:unique-vector}.
\end{enumerate}
Write $c_\sigma(\sigma',T_{i-1})$ and $c_\sigma(\tau,\st_i)$ and
$c_\rho(\rho',S_i)$ for the coefficients on~$\sigma'$ and~$\tau$
and~$\rho'$ in the circuit, shrub, and hedge rim of an
$(i-1)$-face~$\sigma$, an $(i-1)$-stake~$\sigma$, and
an~$i$-face~$\rho$.
\end{defn}

\begin{example}\label{e:hedge}
In the following simplicial complex, choose the hedge $\st_2 =
(S_1,T_2) = (\{bc,cd\},\{abc,bcd\})$ of dimension~$2$.  The shrub and
hedge rim of the stake~$cd$~are
$$%
\begin{array}{c}
\begin{tikzpicture}[scale=1]
\draw[ultra thick,fill=black!30!green] (0,5) -- (1,4) -- (-1,4) -- (0,5);
\draw[ultra thick,fill=black!30!green] (-1,4) -- (-0.5,2.5) -- (1,4) -- (-1,4);
\draw[ultra thick] (-0.5,2.5) -- (1,2.5) -- (1,4) -- (-0.5,2.5);
\draw[line width=5,brown] (-1,4) -- (1,4);
\draw[line width=5,brown] (1,4) -- (-0.5,2.5);
\filldraw (0,5) circle (2pt);
\node[above] at (0,5) {$a$};
\filldraw (-1,4) circle (2pt);
\node[left] at (-1,4) {$b$};
\filldraw (1,4) circle (2pt);
\node [right] at (1,4) {$c$};
\filldraw (-0.5,2.5) circle (2pt);
\node[left] at (-0.5,2.5) {$d$};
\filldraw (1,2.5) circle (2pt);
\node[right] at (1,2.5) {$e$};
\node[right] at (2,4) {$s(cd) = -abc + bcd$};
\node[right] at (2,3) {$r(cd) = ab - ac + bd$.};
\end{tikzpicture}
\end{array}
$$
\end{example}

\begin{remark}\label{r:boundary}
Definition~\ref{d:chain-linked}.\ref{i:chain} allows $\tau \in K_i$,
but in fact the shrub~$s$ in Lemma~\ref{l:chain-linked} only has
nonzero coefficients on faces in~$T_i$, so a stake $\sigma \in
S_{i-1}$ can only be chain-linked to a face~$\tau \in T_i$.
Similarly, Definition~\ref{d:chain-linked}.\ref{i:boundary} allows
$\rho' \in K_i$, but in fact the hedge rim $r(\rho)$ only has nonzero
coefficients on non-stakes, so a face~$\rho$ can only be
boundary-linked to a non-stake $\rho' \in \ol S_i$.  Note,
furthermore, that if the input is already a non-stake $\rho \in \ol
S_i$, then $r(\rho) = \rho$, so only $\rho' = \rho$ is possible, and
$c_\rho(\rho,S_i) = 1$.
\end{remark}

\begin{defn}\label{d:dets}
An \emph{integral structure} on one of the maps in
Definition~\ref{d:hedge} is a free abelian group~$B^\ZZ$ and a
homomorphism $B^\ZZ \!\to\! B$ inducing an isomorphism~\mbox{$B^\ZZ
\otimes \kk \simto\! B$}.  The \emph{square determinants}
$\det(\del_T)^2$ and~$\det(\del_S)^2$ in the presence of integral
structures are the (images in~$\kk$ of the) squares of the
determinants of $B^\ZZ \from \ZZ\{T\}$ and $\ZZ\{S\} \from B^\ZZ$.
\end{defn}

\begin{remark}\label{r:dets}
The square of the determinant of a homomorphism $B^\ZZ \from \ZZ\{T\}$
or $\ZZ\{S\} \from B^\ZZ$ can be calculated using any basis
of~$B^\ZZ$: choosing another basis of~$B^\ZZ$ yields the same
determinant up to a sign that becomes irrelevant upon squaring.
\end{remark}

For CW complexes, hedges and shrubberies yield bases consisting of
faces and boundaries of faces, but only over fields.  Over the
integers, the subgroups generated by these $\ZZ$-linearly independent
sets need not be saturated.  The square determinants measuring the
deviations act as weights in combinatorial formulas for projections
and splittings.

\begin{defn}\label{d:Delta}
Fix a CW complex~$K$.  Define the following integer sums over
shrubberies, stake sets, and hedges using the natural integral
structure~\smash{$\wt B_{i-1}^\ZZ K$} on~$\del_i$:
$$%
  \Delta_i^T K = \sum_{T_i} \det(\del_{T_i})^2,
  \quad
  \Delta_{i-1}^S K = \sum_{S_{i-1}} \det(\del_{S_{i-1}})^2,
  \quad\text{and}\quad
  \Delta_i^\st K = (\Delta_{i-1}^S K) (\Delta_i^T K).
$$
The symbol ``$K$'' may be omitted if the context is clear.  A
field~$\kk$ is \emph{torsionless} for~$K$ if for all~$i$ the
characteristic of~$\kk$ does not divide these numbers or the order of
the torsion subgroup of $\wt C_i^\ZZ K / \wt B_i^\ZZ K$.
\end{defn}

\begin{remark}\label{r:torsion}
The determinants in Definition~\ref{d:Delta} admit combinatorial
interpetations as orders of torsion subgroups of CW subcomplexes
of~$K$ whose facets are determined by the shrubberies and stake sets;
see \cite{catanzaro-chernyak-klein2015}, for example.
\end{remark}

\section{Canonical sylvan resolutions}\label{s:sylvan}

\noindent
The differential in the canonical sylvan resolution is based on
averaging over all hedge choices along lattice paths.  This section
develops notation for Koszul hedges along lattice paths.  Other than
the statement of the main theorem at the end (Theorem~\ref{t:sylvan}),
it contains no results or proofs---only definitions and notation.

\begin{defn}\label{d:hedgerow}
Fix a monomial ideal~$I$ and a saturated decreasing lattice
path~$\lambda$ from $\bb$ to~$\aa$, meaning that adjacent nodes in the
path differ by a standard basis vector of the integer lattice~$\ZZ^n$.
Write $\Lambda(\aa,\bb)$ for the set of such paths.  A \emph{hedgerow}
on~$\lambda$~is
\begin{itemize}
\item%
a stake set $S_i^\bb \subseteq K_i^\bb I$ at the upper endpoint~$\bb$,

\item%
a hedge $\st_i^\cc = (S_{i-1}^\cc, T_i^\cc)$ on the Koszul simplicial
complex~$K^\cc I$ for each interior (i.e., non-endpoint) vertex $\cc$
of~$\lambda$, and

\item%
a shrubbery $T_{i-1}^\aa \subseteq K_{i-1}^\aa I$ at the lower
endpoint~$\aa$.
\end{itemize}
It is convenient to write
$$%
  ST_i^{\,\lambda}
  = (S_{i-1}^\lambda, T_i^\lambda)
$$
for this hedgerow on~$\lambda$, where---to emphasize---this notation
carries an implicit stake set $S_i^\bb$ and shrubbery~$T_{i-1}^\aa$ at
the upper and lower endpoints.  It is also convenient~to~write
$$%
  \bb = \bb_0, \bb_1, \dots, \bb_{\ell-1}, \bb_\ell = \aa
$$
to denote the lattice points on the path $\lambda$, even if sometimes
$\cc$ is written for an otherwise unnamed vertex~$\bb_j$.  The path
$\lambda = (\lambda_1,\dots,\lambda_\ell)$ can be identified with
the~sequence
$$%
  \lambda_j = \bb_{j-1} - \bb_j \text{ for }j = 1,\dots,\ell
$$
of steps in the path, so $\xx^{\lambda_j} \in \{x_1,\dots,x_n\}$ is
one of the $n$~variables.
\end{defn}

\begin{defn}\label{d:DeltaLambda}
Definition~\ref{d:hedge}, Example~\ref{e:CW}, and
Definition~\ref{d:Delta} lead to notations
$$%
\begin{array}{r@{\ }c@{\ }l@{\qquad\quad}r@{\ }c@{\ }l}
  \delta_{i,\bb} &=& \det(\del_{S_i^\bb})
  &
  \Delta_i^{\!\bb} &=& \Delta_i^S K^\bb I,
\\\delta_{i,\bb_j} &=& \det(\del_{S_{i-1}^{\bb_j}})\det(\del_{T_i^{\bb_j}})
  &
  \Delta_i^{\!\bb_j} &=& \Delta_i^\st K^{\bb_j} I
  \qquad\text{ for } j = 1,\dots,\ell-1,
\\\delta_{i,\aa} &=& \det(\del_{T_{i-1}^\aa})
  &
  \Delta_{i-1}^{\!\aa} &=& \Delta_{i-1}^T K^\aa I,
  \qquad\text{and finally}
\\[.75ex]
  \delta_{i,\lambda} &=& \dis\prod_{j=0}^\ell \delta_{i,\bb_j}
  &
  \Delta_{i,\lambda} I &=& \dis\prod_{j=0}^\ell \Delta_i^{\!\bb_j}
\end{array}
$$
for the (images in the field~$\kk$ of the) integer invariants that
take into account hedges in the sequence of Koszul simplicial
complexes of~$I$ along the lattice path~$\lambda$.
\end{defn}

\begin{defn}\label{d:fence}
Fix a lattice path~$\lambda \in \Lambda(\aa,\bb)$.  A \emph{chain-link
fence}~$\phi$ from an $i$-simplex $\tau$ to an $(i-1)$-simplex
$\sigma$ along~$\lambda$ is a choice of hedgerow $ST_i^\lambda$ and a
sequence
$$%
\begin{array}{*{15}{@{}c@{}}}
\\[-3.2ex]
&&\!\!\tau_{\ell-1}\!\!\!&&&&\ \cdots\ &&&&\tau_1&&&&\tau_0\hbox{ --- }\tau\\
&/&&\bs&&/&&\bs&&/&&\bs&&/&\\
\sigma\hbox{ --- }\sigma_\ell&&&&\!\!\sigma_{\ell-1}\!\!\!&&&&\sigma_2&&&&\sigma_1&&
\\[-.2ex]
\end{array}
$$
in which $\tau_j \in
K_i^{\raisebox{.2ex}[6pt][0pt]{$\scriptstyle\bb_j$}} I$ and $\sigma_j
\in K_{i-1}^{\raisebox{.2ex}[6pt][0pt]{$\scriptstyle\bb_j$}} I$ and
\begin{enumerate}
\item[$\scriptstyle\tau_0\,${\footnotesize---}$\scriptstyle{\,\tau}$]%
the simplex
$\tau$ is boundary-linked to~$\tau_0$ via the stake set~$S_i^\bb$;

\item[$\scriptstyle\bs$\ \ \,\,\ ]%
the simplex $\sigma_j \in S_{i-1}^{\raisebox{.2ex}[6pt][0pt]{$\scriptstyle\bb_j$}}$
for $j = 1,\dots,\ell-1$ is a stake chain-linked to~$\tau_j$;

\item[$\scriptstyle/$\ \ \,\,\ ]%
the simplex $\sigma_j$ for $j = 1,\dots,\ell$ equals the facet
$\tau_{j-1} - \lambda_j$ of the simplex~$\tau_{j-1}$;

\item[$\scriptstyle\sigma\,${\footnotesize---}$\scriptstyle{\,\sigma_\ell}$]%
the simplex $\sigma_\ell \in K_{i-1}^\aa$ is cycle-linked to~$\sigma$.
\end{enumerate}
The $i$-simplex $\tau$ and $(i-1)$-simplex $\sigma$ are the
(\emph{initial} and \emph{terminal}) \emph{posts} of the fence~$\phi$.
The set $\Phi(\lambda)$ of chain-link fences along~$\lambda$ has the
subset $\Phi_{\sigma\tau}(\lambda)$ with posts $\tau$ and~$\sigma$.
\end{defn}

\begin{defn}\label{d:weight}
Each chain-link fence edge has a \emph{weight}:
\begin{itemize}\itemsep=1ex
\item%
the boundary-link $\tau_0$\,---\,$\tau$ has weight $\delta_{i,\bb}^2
c_\tau(\tau_0,S{}_i^{\,\bb})$,
\item%
the chain-link
\raisebox{2ex}[0pt][0pt]{$\tau_j\!$}\,\,\raisebox{.5ex}{\tiny$\diagdown$}\,%
\raisebox{-1ex}[0pt][0pt]{$\sigma_j$}
has weight $\delta_{i,\bb_j}^2
c_{\sigma_j}(\tau_j,\st_i{}^{\!\!\bb_j})$,
\item%
the containment\,
\raisebox{-1.25ex}[0pt][0pt]{$\,\sigma_j$}\raisebox{.25ex}{\tiny$\diagup$}\,%
\raisebox{1.25ex}[0pt][0pt]{$\tau_{j-1}$} 
has weight $(-1)^{\sigma_j \subset \tau_{j-1}}$, and
\item%
the cycle-link $\sigma$\,---\,$\sigma_\ell$ has weight
$\delta_{i,\aa}^2\,c_{\sigma_\ell}(\sigma,T_{i-1}^\aa)$.
\end{itemize}
The \emph{weight} of the chain-link fence $\phi$ is the
product~$w_\phi$ of the weights on its edges.
\end{defn}

\begin{remark}\label{r:subordinate}
The notion of chain-link fence~$\phi$ includes the ambient hedgerow
on~$\lambda$.  It is convenient to define~$\phi$ to be
\emph{subordinate} to the given hedgerow~$ST_i^\lambda$, which is
expressed as $\phi = (\sigma, \sigma_\ell, \tau_{\ell-1},
\sigma_{\ell-1}, \dots, \sigma_2, \tau_1, \sigma_1, \tau_0, \tau)
\vdash ST_i^\lambda$.
\end{remark}

\pagebreak

\begin{defn}\label{d:sylvan-matrix}
Fix a monomial ideal~$I$.  The \emph{canonical sylvan homomorphism}
$$%
  \wt C_{i-1} K^\aa I
  \ \stackrel{\ D\ =\ D^{\aa\bb}}\filleftmap\
  \wt C_i K^\bb I
$$
is given by its \emph{sylvan matrix}, whose entry $D_{\sigma\tau}$ for
$\tau \in K_i^\bb I$ and $\sigma \in K_{i-1}^\aa I$ is the sum of the
weights of all chain-link fences from~$\tau$ to~$\sigma$ along all
lattice paths from~$\bb$~to~$\aa$:
$$%
  D_{\sigma\tau}
  =
  \sum_{\lambda \in \Lambda(\aa,\bb)}
  \frac{1}{\Delta_{i,\lambda} I}
  \sum_{\phi \in \Phi_{\sigma\tau}(\lambda)}
  w_\phi.
$$
\end{defn}

\setcounter{sylvan-sec}{\value{section}}
\setcounter{sylvan}{\value{thm}}
\begin{thm}\label{t:sylvan}
Fix a monomial ideal~$I$ and field~$\kk$ that is torsionless for all
Koszul simplicial complexes of~$I$.  The canonical sylvan homomorphism
for each comparable pair $\bb \succ \aa$ of lattice points induces a
homomorphism $\wt Z_{i-1} K^\aa I \from \wt Z_i K^\bb I$ that vanishes
on~\smash{$\wt B_i K^\bb I$}, and hence it induces a well defined
\emph{canonical sylvan homology morphism} $\HH_{i-1} K^\aa I \from
\HH_i K^\bb I$.  The induced homomorphisms
$$%
  \HH_{i-1} K^\aa I \otimes \kk[\xx] (-\aa)
  \,\from\,
  \HH_i K^\bb I \otimes \kk[\xx] (-\bb)
$$
of\hspace{.3ex} $\NN^n$-graded free $\kk[\xx]$-modules constitute a
minimal free resolution of~$I$.
\end{thm}

The proof is at the end of Section~\ref{s:resolutions}.


\section{Examples: sylvan block matrix notation}\label{s:examples}

\begin{example}\label{e:beta1=2}
One of the sticking points for any construction of canonical
minimal~free
\end{example}
\vspace{-2ex}
\begin{wrapfigure}{L}{0.4\textwidth}
  \vspace{-1ex}
  \psfrag{x}{$\!x$}
  \psfrag{y}{$y$}
  \psfrag{z}{$z$}
  \includegraphics[height=2in]{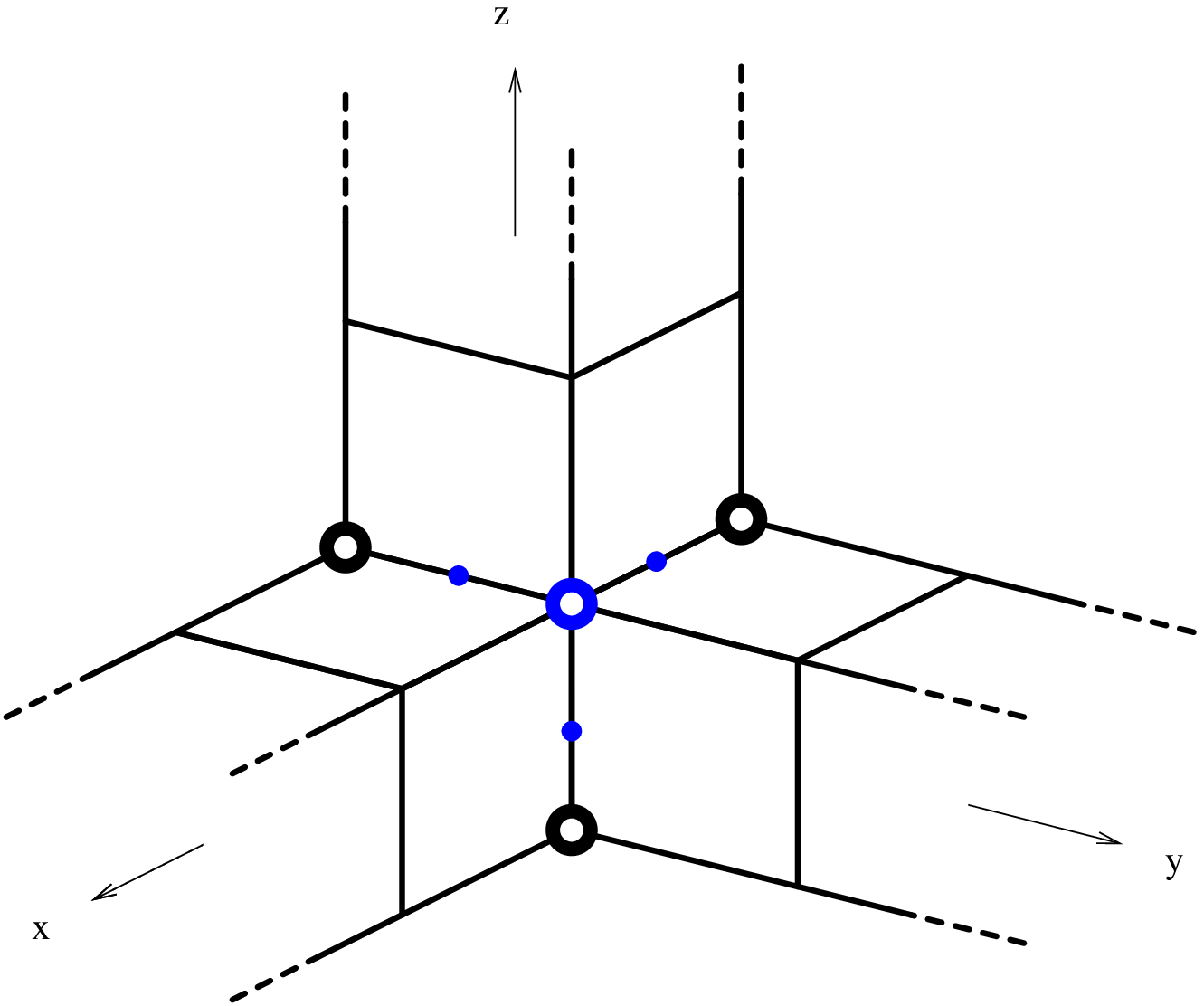}
  \vspace{-1ex}
\end{wrapfigure}
\noindent
resolutions of monomial ideals is $I = \<xy,yz,xz\>$, which has
Betti number $\betti 1{\blu{111}}I = 2$ arising from the Koszul
simplicial complex $K^{\blu{111}} I$ whose facets are the three
vertices
depicted here as little blue dots.  The two linearly independent
syzygies in degree~$\blu{111}$ split equally among the three
generators, which is problematic for any construction that produces
matrices for the differential, because any choice of basis breaks the
symmetry.  The sylvan method avoids this problem by defining the
homomorphism on chains instead of on homology classes.  Here is how it
works for $I = \<xy,yz,xz\>$.

Start by computing, say, the sylvan matrix~$D^{110,\blu{111}}$.  There
is only one lattice path $\lambda \in \Lambda(110,{\blu{111}})$, since
the length is~$1$.  For this path, the stake set at the initial
post~$\blu{111}$ is forced to be $\blu S_0^{111} = \nothing$ because
$\blu K^{111} I$ has dimension~$0$ and hence no faces of dimension~$1$
to take the boundary~of.  The shrubbery at the terminal post~$110$ is
forced to be the empty set of faces, which is written $T_{-1}^{110} =
\{\}$ so as not to confuse it with the set $\{\nothing\}$ consisting
of the empty face.  As there is only one hedgerow on~$\lambda$, and
all of the square determinants~$\delta^2$ equal~$1$ because the number
of vertices is too small for torsion in homology, $\Delta_{0,\lambda}
= 1$.  To construct a chain-link fence along~$\lambda$, only one
terminal post~$\sigma = \nothing$ is available.  Trying $\blu x$ as
initial post yields no fences: the boundary link {\blu $x$\,---\,$x$}
is forced, because $\blu x$ is the unique non-stake~$\blu \tau'$ in
$\blu K^{111}I$ such that $\blu x - \tau'$ is a boundary, but then $z
= \lambda_1$ is not a face of~$\blu x$, so the fence has no
continuation.
$$%
\begin{array}{*{3}{@{}c@{}}}
\\[-2.2ex]
   &&{\blu x \edgehoriz 1 x}
\\ &\edgeup 0&
\\\phantom{\nothing \edgehoriz 1 \nothing}
\\[-.2ex]
\end{array}
\qquad\qquad
\begin{array}{*{3}{@{}c@{}}}
\\[-2.2ex]
   &&{\blu y \edgehoriz 1 y}
\\ &\edgeup 0&
\\\phantom{\nothing \edgehoriz 1 \nothing}
\\[-.2ex]
\end{array}
\qquad\qquad
\begin{array}{*{3}{@{}c@{}}}
\\[-2.2ex]
   &&{\blu z \edgehoriz 1 z}
\\ &\edgeup 1&
\\\nothing \edgehoriz 1 \nothing
\\[-.2ex]
\end{array}
$$
Similarly, no fence has initial post~$\blu y$.  But the initial
post~$\blu z$ yields one fence, as depicted.  The conclusion is that
$\Phi_{\nothing,x} = \nothing = \Phi_{\nothing,y}$ and
$|\Phi_{\nothing,z}| = 1$, with all of the edge coefficients and
torsion numbers equal to~$1$.  This determines the top row of the
matrix\vspace{-1ex}
$$%
\begin{array}{c}
   \HH_{-1} K^{110} \otimes \<xy\>
\\ \oplus
\\ \HH_{-1} K^{101} \otimes \<xz\>
\\ \oplus
\\ \HH_{-1} K^{011} \otimes \<yz\>
\end{array}
\xleftarrow{
\monomialmatrix
  {\nothing\\\nothing\\\nothing}
  {\begin{array}{@{\,}c@{\,}ccc@{\,}c@{\,}}
    {}  &\blu x &\blu y &\blu z 
  \\{}[ & 0 & 0 & 1 & ]
  \\{}[ & 0 & 1 & 0 & ]
  \\{}[ & 1 & 0 & 0 & ]
  \end{array}}
  {\\\\\\}
}
{\blu \HH_0 K^{111} \otimes \<xyz\>}
$$
in which the other two rows are determined by symmetry.  The blocks in
this matrix are the sylvan matrices $D^{110,\blu 111}$, $D^{101,\blu
111}$, and $D^{011,\blu 111}$.  Its rows and columns are labeled by
the faces of the corresponding Koszul simplicial complexes.  The
cycles $\blu x - y$, $\blu z - x$, and $\blu y - z$ correspond to the
column vectors
$$%
{\blu\left[\begin{array}{c} 1\\-1\\ 0\end{array}\right]},
\qquad
{\blu\left[\begin{array}{c}-1\\ 0\\ 1\end{array}\right]},
\quad\text{and}\quad
{\blu\left[\begin{array}{c} 0\\ 1\\-1\end{array}\right]}.
$$
The image of, say, $\blu x - y$ is the $\NN^3$-degree $111$ element
$x \cdot \nothing - y \cdot \nothing$ that is $x$ times the free
generator of $\HH_{-1} K^{011} \otimes \<yz\>$ minus $y$ times the
free generator of $\HH_{-1} K^{101} \otimes \<xz\>$.  Looking back at
the staircase drawn at the beginning of this Example, the face~$\blu
x$ of the cycle $\blu x - y$ has moved back parallel to the $x$-axis
to become $x \cdot \nothing$, and $\blu -y$ has moved back parallel to
the $y$-axis to become $-y \cdot \nothing$.  This behavior is
fundamental to the sylvan construction: faces move back in directions
parallel to axes, picking up the corresponding monomial coefficients
as they go.

\begin{conv}\label{conv:block-matrix}
The general block matrix notational device illustrated by
Example~\ref{e:beta1=2} works as follows.  Choose an order in which to
list the $\NN^n$-degrees $\aa$ and~$\bb$ of nonzero Betti numbers
$\betti i\aa I$ and $\betti {i+1}\bb I$ in homological stages~$i$
and~$i+1$, say $\aa_1,\aa_2,\dots$ and $\bb_1,\bb_2,\ldots$.  The $pq$
block is the matrix $D^{\aa_p \bb_q}$, which takes chains in $\wt C_i
K^{\bb_q} I$, thought of as column vectors with entries indexed by the
$i$-simplices in $K_{i\!\!}{}^{\bb_q}{} I$, to chains in $\wt C_{i-1}
K^{\aa_p} I$, thought of as column vectors with entries indexed by the
$(i-1)$-simplices in $K_{i-1}^{\aa_p} I$.  The orderings on the
$\NN^n$-degrees $\aa_p$ and~$\bb_q$ are depicted by writing ordered
direct sums vertically.  The orderings on the simplices are depicted
by labeling the rows and columns of each~block.
\end{conv}

\begin{example}\label{e:full-sylvan-res}
It might be helpful to see a more complicated canonical sylvan
resolution, with multiple lattices paths from $\bb$ to~$\aa$ and
chain-link fences along a fixed~lattice
\end{example}\vspace{-2.25ex}
\begin{wrapfigure}{L}{0.5\textwidth}
  \vspace{-1.15ex}
  \psfrag{x}{$\!x$}
  \psfrag{y}{$y$}
  \psfrag{z}{$z$}
  \includegraphics[height=1.62in]{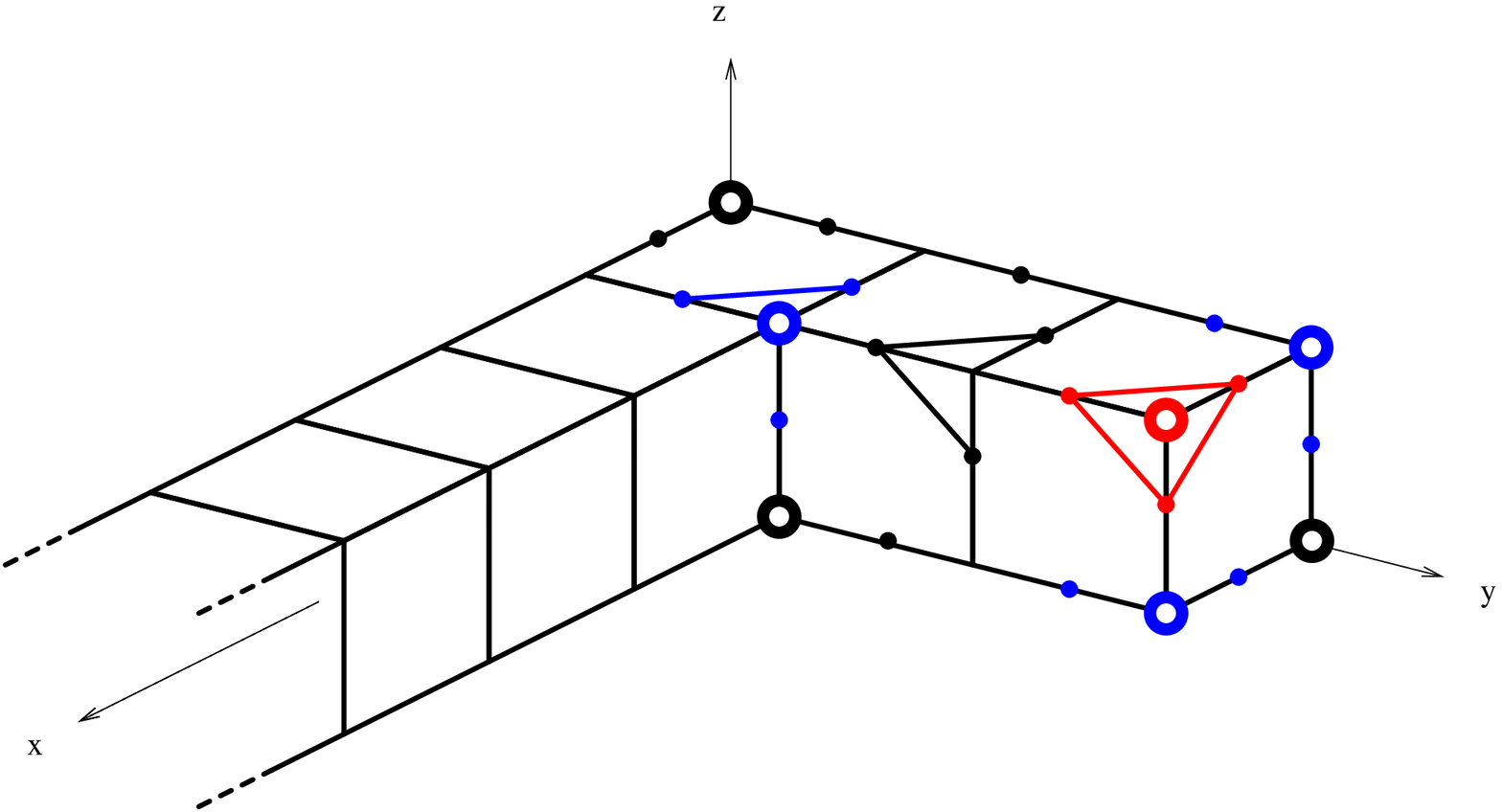}
  \vspace{-2ex}
\end{wrapfigure}
\noindent
path.  Alas, the square determinants are all~$1$ in this example,
since nontrivial such invariants can only occur with enough vertices
that writing down the full canonical sylvan resolution would be an
ineffective use of space.

Let $I = \<xy,y^3,z\>$, whose staircase is depicted here.
The three large black dots are the generators, the three large blue
dots are the first syzygies, and the large red dot is the second
syzygy.  The little dots and edges behind and below each of the large
dots form its Koszul simplicial complex: a triangle (a~nontrivial loop
in~\smash{$\red \HH_1$}) for the second syzygy; a disconnected union of two
vertices or an edge and a vertex (nontrivial~\smash{$\blu \HH_0$}) for
the first syzygies; and just the empty face
(nontrivial~\smash{$\HH_{-1}$}; not drawn) for the generators.

The canonical sylvan resolution of~$I$, notated as per
Convention~\ref{conv:block-matrix}, is as follows.
{\foot\vspace{-1.5ex}
$$%
\begin{array}{@{}c@{\,}}
   \HH_{-1} K^{110} \!\otimes\! \<xy\>
\\ \oplus
\\ \HH_{-1} K^{030} \!\otimes\! \<y^3\>
\\ \oplus
\\ \HH_{-1} K^{001} \!\otimes\! \<z\>
\end{array}
\xleftarrow{\!\!\!
\monomialmatrix
  {\nothing\\\nothing\\\nothing}
  {\begin{array}{@{\,}c@{\ }c@{\ }c@{\,}}
      \blu x\ y\ z    &\blu x\ y    &\blu y\ z   
  \\{}\![\,0\ 0\ 1\,] &  [\,0\ 1\,] &  [\,0\ 0\,]\!
  \\{}\![\,0\ 0\ 0\,] &  [\,1\ 0\,] &  [\,0\ 1\,]\!
  \\{}\![\,1\ 1\ 0\,] &  [\,0\ 0\,] &  [\,1\ 0\,]\!
  \end{array}}
  {\\\\\\}
\!\!\!}
{\blu
\begin{array}{@{}c@{\,}}
   \HH_0 K^{111} \!\otimes\! \<xyz\>
\\ \oplus
\\ \HH_0 K^{130} \!\otimes\! \<xy^3\>
\\ \oplus
\\ \HH_0 K^{031} \!\otimes\! \<y^3z\>
\end{array}
}
\xleftarrow{\!\!\!
\monomialmatrix
  {\blu x\\\blu y\\\blu z\\[.2ex]\blu x\\\blu y\\[.2ex]\blu y\\\blu z}
  {\begin{array}{@{}l@{\ }c@{\ }r@{}}
    \red\quad\, zy &\red\ \ yx &\red xz\ \,
    \\
    \!
    \left[
    \begin{array}{@{}r@{}}
	   \nicefrac49
	\\ \nicefrac19
	\\ \nicefrac{-5}9
    \end{array}
    \right.
    &
    \begin{array}{@{}r@{}}
	   \nicefrac59
	\\ \nicefrac{-1}9
	\\ \nicefrac{-4}9
    \end{array}
    &
    \left.
    \begin{array}{@{}c@{}}
	  \ \ 0\
	\\\ \ 0\
	\\\ \ 0\
    \end{array}
    \right]
    \!
    \\[3ex]
    \hspace{-.1ex}
    \left[
    \begin{array}{@{}r@{}}
	  \nicefrac{-1}2\!\!
	\\\nicefrac12\!\!
    \end{array}
    \right.
    &
    \begin{array}{@{\!\!}r@{}}
	  \ \ 0
	\\\ \ 0
    \end{array}
    &
    \left.
    \begin{array}{@{}r@{}}
	   \!\!\nicefrac{-1}2\,
	\\ \!\!\nicefrac12\,
    \end{array}
    \right]
    \hspace{-.1ex}
    \\
    \hspace{-.1ex}
    \left[
    \begin{array}{@{}r@{}}
	  \ \ \ 0
	\\\ \ \ 0
    \end{array}
    \right.
    &
    \begin{array}{@{}r@{}}
	   \nicefrac{-1}2\
	\\ \nicefrac12\
    \end{array}
    &
    \left.
    \begin{array}{@{}r@{}}
	   \!\!\nicefrac{-1}2\,
	\\ \!\!\nicefrac12\,
    \end{array}
    \right]
    \hspace{-.1ex}
    \end{array}}
  {\\\\\\[.6ex]\\\\[.6ex]\\\\}
\!\!\!}
{\red \HH_1 K^{131} \!\otimes\! \<xy^3z\>}
$$
}%

\noindent
Before getting to hedges and fences that compute the sylvan matrix
entries, several features are to be noted.  The column vector $\red
(1,1,1)^\top$ generates $\red \HH_1 K^{131}$ and maps to the sum $\blu
(1,0,-1)^\top \oplus (-1,1)^\top \oplus (-1,1)^\top$ in the
$\NN^3$-degree~$131$ component of the middle free module.  (The
relevant fractions magically either cancel or sum to~$1$.)  This
statement should be viewed geometrically on the staircase: the cycle
that is the red triangle $\red zy + yx + xz$ maps to the difference of
the two vertices closest to it in each of the blue simplicial
complexes.  At $\blu 111$ this is $\blu x - z$; at $\blu 130$ this is
$\blu y - x$; at $\blu 031$ this is $\blu z - y$.
The direct sum of these three blue cycles yields a vector of
length~$9$ that lies in the kernel of the $3 \times 9$ matrix---that
is, the ($3$~block)~$\times$~($3$~block) matrix.

\begin{excise}{%
  Let us start the sylvan matrix computations with $D^{001,\blu 031}$.
  There is only one lattice path $\lambda \in \Lambda(001,{\blu 031})$,
  namely
  $$%
  \begin{array}{@{}r@{\qquad\quad}c@{\,}c@{\,}c@{\,}c@{\,}c@{\,}c@{\,}c@{\,}c@{}}
     &    &\qquad\qquad&        &\qquad\qquad&     &\qquad\qquad\\[-2.5ex]
  \\[-2.5ex]\lambda:&
      001    &\fillbar&     011    &\fillbar&   021   &\fillbar&     \blu 031
  \\\st_0^{\,\lambda}:&
  \mko{T_{-1} = \{\}}&&\mko{T_0 = \{y\}}&&\mko{T_0 = \{y\}}&&\mko{\blu S_0 = \nothing}
  \\ &         &&\mko{S_{-1}=\{\nothing\}}&&\mko{S_{-1}=\{\nothing\}}
  \end{array}
  $$
  whose unique dimension~$0$ hedgerow is indicated underneath.  This
  uniqueness implies
  $$%
    \Delta_{0,\lambda} = 1
  $$
  There is furthermore only one chain-link fence along~$\lambda$: the
  initial post~$\blu x$ is barred because $\blu x$ is not a face
  of~$\blu K^{031} I$; the initial post~$\blu z$ is boundary-linked only
  to~$\blu z$ (because $\blu K^{031} I$ has no boundaries in
  dimension~$0$) and no continuation is possible because $y$ is not a
  face of~$\blu z$; and the initial post~$\blu y$ yields alternating
  faces~$y$~and~$\nothing$.
  $$%
  \begin{array}{*{3}{@{}c@{}}}
  \\[-2.2ex]
     &&{\blu z \edgehoriz 1 z}
  \\ &\edgeup 0&
  \\\phantom{\nothing \edgehoriz 1 \nothing}
  \\[-.2ex]
  \end{array}
  \qquad\qquad
  \begin{array}{*{11}{@{}c@{}}}
  \\[-2.2ex]
     &         &y&           &        &         &y&           &        &         &
    \blu y \edgehoriz 1 y
  \\ &\edgeup 1& &\edgedown 1&        &\edgeup 1& &\edgedown 1&        &\edgeup 1&
  \\ \nothing\edgehoriz 1\nothing
     &         & &           &\nothing&         & &           &\nothing&         &
  \\[-.2ex]
  \end{array}
  $$
  Multiplying the coefficients in this fence to get $D^{001,{\blu
  031}}_{\nothing,\blu y} = 1$ completes the sylvan matrix $D^{001,{\blu
  031}} = [0\ 1\ 0]$ in the lower-right corner of $F_0 \from {\blu
  F_1}$.  Similar computations---but easier, because the lattice paths
  are shorter---yield the sylvan matrices
  $$%
  \begin{array}{ccc}
  D^{110,{\blu 111}}=[0\ 0\ 1]&D^{110,{\blu 130}}=[0\ 1\ 0]\\[.5ex]
                              &D^{030,{\blu 130}}=[1\ 0\ 0]&D^{030,{\blu 031}}=[0\ 0\ 1]
  \end{array}
  $$
  in the top two rows of $F_0 \from {\blu F_1}$.
}\end{excise}%
For a sample sylvan matrix computation for $F_0 \from {\blu F_1}$,
consider the $1 \times 3$ sylvan matrix $D^{001,{\blu 111}}$ in the
lower-left corner.  It involves two lattice paths
in~$\Lambda(001,{\blu 111})$:
$$%
\begin{array}{@{}r@{\qquad\quad}c@{\,}c@{\,}c@{\,}c@{\,}c@{\,}c@{}}
   &    &\qquad\qquad&        &\qquad\qquad&     &\qquad\qquad\\[-2.5ex]
\\[-2.5ex]\lambda:&
    001    &\fillbar&     011    &\fillbar&   \blu 111
\\\st_0^{\,\lambda}:&
\mko{T_{-1} = \{\}}&&\mko{T_0 = \{y\}}&&\mko{{\blu S_0 = \{x\}}}
\\ &         &&\mko{S_{-1}=\{\nothing\}}&&\mko{\phantom{S_0\!}\text{ or }{\blu \{y\}}}
\end{array}
$$
and
$$%
\begin{array}{@{}r@{\qquad\quad}c@{\,}c@{\,}c@{\,}c@{\,}c@{\,}c@{}}
\\[-4ex]
   &    &\qquad\qquad&        &\qquad\qquad&     &\qquad\qquad\\[-2.5ex]
\\[-2.5ex]\lambda':&
    001    &\fillbar&     101    &\fillbar&   \blu 111
\\\st_0^{\,\lambda}:&
\mko{T_{-1} = \{\}}&&\mko{T_0 = \{x\}}&&\mko{{\blu S_0 = \{x\}}}
\\ &         &&\mko{S_{-1}=\{\nothing\}}&&\mko{\phantom{S_0\!}\text{ or }{\blu \{y\}}}
\end{array}
$$
both of which have two choices for the initial~$\blu S_0$ and hence
have
$$%
  \Delta_{0,\lambda'} = \Delta_{0,\lambda} = 1 \cdot 1 \cdot {\blu 2} = 2.
$$
For~$\lambda$, regardless of whether $\blu S_0 = \{x\}$ or $\blu S_0 =
\{y\}$, the initial post~$\blu z$ is boundary-linked only to~$\blu z
\in \ol S_0$, which does not contain~$y$ and hence yields no fences
along~$\lambda$.  The initial post~$\blu y$ is boundary-linked
to~$\blu y$ if $\blu S_0 = \{x\}$ and to~$\blu x$ if $\blu S_0 =
\{y\}$.  However, the case of $\blu \tau_0\,\hbox{---}\,\tau =
y\,\hbox{---}\,y$ has no continuation because~$x$ is not a face
of~$\blu y$.  In contrast, the case of $\blu \tau_0\,\hbox{---}\,\tau
= x\,\hbox{---}\,y$ yields a valid fence.  The same logic shows that
$\blu \tau_0\,\hbox{---}\,\tau = y\,\hbox{---}\,x$ when $\blu S_0 =
\{x\}$ has no continuation but $\blu \tau_0\,\hbox{---}\,\tau =
x\,\hbox{---}\,x$ when $\blu S_0 = \{y\}$ yields a valid fence.  These
possibilities for $\lambda = 001\,\hbox{---}\,011\,\hbox{---}\,{\blu
111}$ are summarized as follows.
$$%
\!\!\!\!\!
\begin{array}{*{3}{@{}c@{}}}
\\[-2.2ex]
   &&{\blu z \edgehoriz 1 z}
\\ &\edgeup 0&
\\\phantom{\nothing}
\\[1ex]
   &&\mko{{\blu S_0 = \{x\}}}
\\ &&\mko{\phantom{S_0\!}\text{ or }{\blu \{y\}}}
\\[-.2ex]
\end{array}
\qquad
\begin{array}{*{3}{@{}c@{}}}
\\[-2.2ex]
   &&{\blu y \edgehoriz 1 y}
\\ &\edgeup 0&
\\\phantom{\nothing}
\\[1ex]
   &&\mko{\blu S_0 = \{x\}\ }
\\ &&\mko{\phantom{S_0\!\text{ or }{\blu \{y\}}}}
\\[-.2ex]
\end{array}
\qquad
\begin{array}{*{7}{@{}c@{}}}
\\[-2.2ex]
               &         &y&           &        &         &\blu x \edgehoriz 1 y
\\             &\edgeup 1& &\edgedown 1&        &\edgeup 1&
\\ \nothing
   \edgehoriz 1
       \nothing&         & &           &\nothing&         &
\\[1ex]
   &&&&&&\mko{\blu S_0 = \{y\}\ \ }
\\ &&&&&&\mko{\phantom{\blu S_0 = \{y\}\ \ }}
\\[-.2ex]
\end{array}
\qquad
\begin{array}{*{3}{@{}c@{}}}
\\[-2.2ex]
   &&{\blu y \edgehoriz 1 x}
\\ &\edgeup 0&
\\\phantom{\nothing}
\\[1ex]
   &&\mko{\blu S_0 = \{x\}\ }
\\ &&\mko{\phantom{S_0\!\text{ or }{\blu \{y\}}}}
\\[-.2ex]
\end{array}
\qquad
\begin{array}{*{7}{@{}c@{}}}
\\[-2.2ex]
               &         &y&           &        &         &\blu x \edgehoriz 1 x
\\             &\edgeup 1& &\edgedown 1&        &\edgeup 1&
\\ \nothing
   \edgehoriz 1
       \nothing&         & &           &\nothing&         &
\\[1ex]
   &&&&&&\mko{\blu S_0 = \{y\}\ \ }
\\ &&&&&&\mko{\phantom{\blu S_0 = \{y\}\ \ }}
\\[-.2ex]
\end{array}
$$
The lattice path~$\lambda$ therefore contributes $\frac12[1\ 1\ 0]$ to
$D^{001,{\blu 111}}$.
\begin{excise}{%
  The calculation for the other path $\lambda' =
  001\,\hbox{---}\,101\,\hbox{---}\,{\blu 111} \in \Lambda(001,{\blu
  111})$ is obtained by swapping the roles of~$x$ and~$y$ throughout; it
  also contributes $\frac12[1\ 1\ 0]$, so $D^{001,{\blu 111}} =
  \frac12[1\ 1\ 0] + \frac12[1\ 1\ 0] = [1\ 1\ 0]$.

  It remains to compute ${\blu F_1} \from {\red F_2}$.  Start with
  $D^{{\blu 130},{\red 131}}$.  There is only one lattice path $\lambda
  \in \Lambda({\blu 130},{\red 131})$, namely $\lambda = {\blu
  130}\,\hbox{---}\,{\red 131}$.  Since $\red K^{111}$ has
  dimension~$1$, it has no boundaries in dimension~$1$ and hence $\red
  S_1^{131} = \nothing$ is forced; this remains true for the entire
  computation of~$D^{{\blu 130},{\red 131}}$.  On the other hand, every
  vertex of~$\blu K^{130}$ has boundary~$\blu\nothing$ and hence is a
  valid choice of~$\blu T_0^{130}$; this reasoning holds more generally
  for~$T_0$ in any nonempty CW~complex.  Consequently
  $$%
    \Delta_{0,{\blu 130}-{\red 131}}
    =
    {\blu 2} \cdot {\red 1}
    =
    2.
  $$
  For chain-link fences, the initial boundary-link $\red
  yx\,\hbox{---}\,yx$ has no continuation because $\blu z$~is not a face
  of~$\red yx$; this explains the column of zeros in the middle block of
  ${\blu F_1} \from {\red F_2}$.  The initial boundary-link $\red
  xz\,\hbox{---}\,xz$ yields $\blu x$ with coefficient~$-1$; when $\blu
  T_0 = \{x\}$ the $\blu T_0$-circuit of~$\blu x$ is $\zeta_{\blu
  T_0}({\blu x}) = {\blu x - x} = {\blu 0}$; but when $\blu T_0 = \{y\}$
  the $\blu T_0$-circuit of~$\blu x$ is $\zeta_{\blu T_0}({\blu x}) =
  {\blu x - y}$, so two chain-link fences result, one in $\Phi_{{\blu
  x},{\red xz}}$ with coefficient~${\blu 1} \cdot -1 \cdot {\red 1} =
  -1$ and one in $\Phi_{{\blu y},{\red xz}}$ with coefficient ${\blu -1}
  \cdot -1 \cdot {\red 1} = 1$.  This explains the right-hand column in
  the middle block of ${\blu F_1} \from {\red F_2}$.  The left-hand
  column is similar, with $\phi \in \Phi_{{\blu x},{\red zy}}$ having
  coefficient~$w_\phi = {\blu -1} \cdot 1 \cdot {\red 1} = 1$ and $\phi
  \in \Phi_{{\blu y},{\red zy}}$ having coefficient $w_\phi = {\blu 1}
  \cdot 1 \cdot {\red 1} = 1$.
  $$%
  \begin{array}{*{3}{@{}c@{}}}
  \\[-2.2ex]
     &&{\red yx \edgehoriz 1 yx}
  \\ &\edgeup 0&
  \\\phantom{\nothing}
  \\[1ex]
     &&\mko{{\blu T_0 = \{x\}}\quad}
  \\ &&\mko{\phantom{S_0\!}\text{ or }{\blu \{y\}}\quad}
  \\[-.2ex]
  \end{array}
  \qquad
  \begin{array}{*{3}{@{}c@{}}}
  \\[-2.2ex]
     &&{\red xz \edgehoriz 1 xz}
  \\ &\edgeup{-1}&
  \\ \blu \phantom{y} \edgehoriz 0 x
  \\[1ex]
     &\mko{{\blu T_0 = \{x\}}\ }
  \\ &\mko{\phantom{S_0\!\text{ or }{\blu \{y\}}}}
  \\[-.2ex]
  \end{array}
  \qquad
  \begin{array}{*{3}{@{}c@{}}}
  \\[-2.2ex]
     &&{\red xz \edgehoriz 1 xz}
  \\ &\edgeup{-1}&
  \\ \blu x \edgehoriz 1 x
  \\[1ex]
     \mko{\quad\ {\blu T_0 = \{y\}}}
  \\ \mko{\phantom{S_0\!\text{ or }{\blu \{x\}}}}
  \\[-.2ex]
  \end{array}
  \qquad
  \begin{array}{*{3}{@{}c@{}}}
  \\[-2.2ex]
     &&{\red xz \edgehoriz 1 xz}
  \\ &\edgeup{-1}&
  \\ \blu y \edgehoriz{-1} x
  \\[1ex]
     \mko{\quad\ {\blu T_0 = \{y\}}}
  \\ \mko{\phantom{S_0\!\text{ or }{\blu \{x\}}}}
  \\[-.2ex]
  \end{array}
  $$
  $$%
  \phantom{
  \begin{array}{*{3}{@{}c@{}}}
  \\[-2.2ex]
     &&{\red yx \edgehoriz 1 yx}
  \\ &\edgeup 0&
  \\\phantom{\nothing}
  \\[1ex]
     &&\mko{{\blu T_0 = \{x\}}\quad}
  \\ &&\mko{\phantom{S_0\!}\text{ or }{\blu \{y\}}\quad}
  \\[-.2ex]
  \end{array}
  }
  \qquad
  \begin{array}{*{3}{@{}c@{}}}
  \\[-2.2ex]
     &&{\red zy \edgehoriz 1 zy}
  \\ &\edgeup 1&
  \\ \blu \phantom{y} \edgehoriz 0 y
  \\[1ex]
     &\mko{{\blu T_0 = \{y\}}\ }
  \\ &\mko{\phantom{S_0\!\text{ or }{\blu \{y\}}}}
  \\[-.2ex]
  \end{array}
  \qquad
  \begin{array}{*{3}{@{}c@{}}}
  \\[-2.2ex]
     &&{\red zy \edgehoriz 1 zy}
  \\ &\edgeup 1&
  \\ \blu y \edgehoriz 1 y
  \\[1ex]
     \mko{\quad\ {\blu T_0 = \{x\}}}
  \\ \mko{\phantom{S_0\!\text{ or }{\blu \{x\}}}}
  \\[-.2ex]
  \end{array}
  \qquad
  \begin{array}{*{3}{@{}c@{}}}
  \\[-2.2ex]
     &&{\red xz \edgehoriz 1 xz}
  \\ &\edgeup{-1}&
  \\ \blu x \edgehoriz{-1} y
  \\[1ex]
     \mko{\quad\ {\blu T_0 = \{x\}}}
  \\ \mko{\phantom{S_0\!\text{ or }{\blu \{x\}}}}
  \\[-.2ex]
  \end{array}
  $$
  
  The calculation of $D^{{\blu 031},{\red 131}}$ is the same as
  $D^{{\blu 130},{\red 131}}$ after cyclic permutation of the variables
  $x \mapsto y \mapsto z$.  Note: this symmetry argument applies to the
  relative locations of the Koszul simplicial complexes at~$\aa$
  and~$\bb$, not to the absolute positions of~$\aa$ and~$\bb$.
}\end{excise}%

For a sample sylvan matrix computation for ${\blu F_1} \from {\red
F_2}$, compute $D^{{\blu 111},{\red 131}}$ using the sole lattice path
$\lambda \in \Lambda({\blu 111},{\red 131})$, namely
$$%
\begin{array}{@{}r@{\qquad\quad}c@{\,}c@{\,}c@{\,}c@{\,}c@{\,}c@{}}
   &    &\qquad\qquad&        &\qquad\qquad\\[-2.5ex]
\\[-2.5ex]\lambda:&
  \blu 111 &\fillbar&     121    &\fillbar& \red 131 
\\\st_0^{\,\lambda}:&
\mko{\blu T_0 = \{x\}}&&\mko{T_1 = \{zy,yx\}}&&\mko{\red S_1 = \nothing}
\\                  &
\mko{\phantom{T_0\!}\text{ or }{\blu \{y\}}}&&\mko{S_0=\{y,z\}}
\\                  &
\mko{\phantom{T_0\!}\text{ or }{\blu \{z\}}}&&\mko{\phantom{S_0}\!\text{ or }\{x,z\}}
\\                  &
\mko{\phantom{T_0\!\text{ or }{\blu \{z\}}}}&&\mko{\phantom{S_0}\!\text{ or }\{x,y\}}
\end{array}
$$
The vector space of homological stage~$0$ boundaries in~$K^{121}$ has
dimension~$2$, so any pair of vertices is a stake set because no
single vertex is a boundary.  Thus
$$%
  \Delta_{1,{\blu 111}-{\red 131}}
  =
  {\blu 3} \cdot 3 \cdot {\red 1}
  =
  9.
$$
The initial boundary-link $\red xz\,\hbox{---}\,xz$ has no
continuation because $y$ is not a face of~$\red xz$; this explains the
zero column in the top block of ${\blu F_1} \from {\red F_2}$.

The initial boundary-link $\red yx\,\hbox{---}\,yx$ yields a much more
interesting computation.  When the stake set $S_0 = \{y,z\}$ is
selected at the interior lattice point~$121$, the facet~$x$ of~$\red
yx$ is forced by
Definition~\ref{d:fence}.\raisebox{.3ex}{\hspace{.1ex}\foot/}, but $x$
is not a stake, so no continuation is possible.  In contrast, when
$S_0 = \{x,z\}$ is selected, four chain-link fences result: two from
the choice of $\blu T_0 = \{y\}$, because the circuit of~$\blu x$ with
respect to the shrubbery~$\blu \{y\}$ is $\blu x - y$, so each of the
terms $\blu x$ and $\blu -y$ contributes a cycle-link; and two from
$\blu T_0 = \{z\}$ for a similar reason with~$\blu z$ in place
of~$\blu y$.  To save space, these four fences are
drawn~as~a~single~fork.
$$%
\begin{array}{*{3}{@{}c@{}}}
\\[-2.2ex]
   &&{\red yx \edgehoriz 1 yx}
\\ &\edgeup 1&
\\ x \makebox[0pt][l]{$\ \not\in S_0$}
\\[1ex]
   \mko{\qquad\qquad S_0 = \{y,z\}}
\\[-.2ex]
\end{array}
\qquad
\qquad
\qquad
\begin{array}{*{7}{@{}c@{}}}
\\[-2.2ex]
    &         &yx&           &   &         &\red yx \edgehoriz 1 yx
\\  &\edgeup 1&  &\edgedown 1&   &\edgeup 1&
\\
\blu
\begin{array}{@{}r@{\ }c@{\ \,}c@{\,}c@{}}
\\[-6.3ex]
  & x\!\!\!
\\[-1.3ex]
  \blu T_0 = \{y\}\colon
  & & \diagdown\hspace{.1ex}\raisebox{1ex}{\tiny\mkl{\hspace{-2ex}\scriptstyle 1}}
\\[-1.7ex]
  & y \raisebox{.75ex}{\tiny\mkl{\,\scriptstyle -1}} & \fillbar
\\[-1.6ex]
  &                                                  &          & x
\\[-1.6ex]
  & x\raisebox{-.5ex}{\tiny\mkl{\ \ \scriptstyle 1}} & \fillbar
\\[-1.7ex]
  \blu T_0 = \{z\}\colon
  & & \diagup\hspace{.1ex}\raisebox{-.5ex}{\tiny\mkl{\hspace{-2.5ex}\scriptstyle -1}}
\\[-1.5ex]
  & z\!\!\!
\\[-3.6ex]
\end{array}
  &           &  &             & x &         &
\\[1ex]
  & &&&&\mko{S_0 = \{x,z\}}
\\[-.2ex]
\end{array}
$$
When $S_0 = \{x,y\}$ is selected, eight chain-link fences result, by
reasoning as in the four fences just constructed.  (In fact, the first
four fences, depicted in the left-hand diagram here, have exactly the
same sequences of simplices and coefficients as the four fences just
produced.  They count separately because they are subordinate to a
different hedgerow.)  The reason why there are twice as many is that
the shrub of~$x$ for the hedge $\st_1 = (\{x,y\}, \{zy,yx\})$ is $zy +
yx$, the unique path in $K^{121}$ that joins the non-stake~$z$ to the
stake~$x$.  Thus the fence with $S_0 = \{x,y\}$ bifurcates at~$x$
into~$yx$~and~$zy$.
$$%
\begin{array}{*{7}{@{}c@{}}}
\\[-2.2ex]
    &         &yx&           &   &         &\red yx \edgehoriz 1 yx
\\  &\edgeup 1&  &\edgedown 1&   &\edgeup 1&
\\
\blu
\begin{array}{@{}r@{\ }c@{\ \,}c@{\,}c@{}}
\\[-6.3ex]
  & x\!\!\!
\\[-1.3ex]
  \blu T_0 = \{y\}\colon
  & & \diagdown\hspace{.1ex}\raisebox{1ex}{\tiny\mkl{\hspace{-2ex}\scriptstyle 1}}
\\[-1.7ex]
  & y \raisebox{.75ex}{\tiny\mkl{\,\scriptstyle -1}} & \fillbar
\\[-1.6ex]
  &                                                  &          & x
\\[-1.6ex]
  & x\raisebox{-.5ex}{\tiny\mkl{\ \ \scriptstyle 1}} & \fillbar
\\[-1.7ex]
  \blu T_0 = \{z\}\colon
  & & \diagup\hspace{.1ex}\raisebox{-.5ex}{\tiny\mkl{\hspace{-2.5ex}\scriptstyle -1}}
\\[-1.5ex]
  & z\!\!\!
\\[-3.6ex]
\end{array}
  &           &  &             & x &         &
\\[1ex]
  & &&&&\mko{S_0 = \{x,y\}}
\\[-.2ex]
\end{array}
\qquad
\begin{array}{*{7}{@{}c@{}}}
\\[-2.2ex]
    &           &zy&           &   &         &\red yx \edgehoriz 1 yx
\\  &\edgeup{-1}&  &\edgedown 1&   &\edgeup 1&
\\
\blu
\begin{array}{@{}r@{\ }c@{\ \,}c@{\,}c@{}}
\\[-6.3ex]
  & z\!\!\!
\\[-1.3ex]
  \blu T_0 = \{x\}\colon
  & & \diagdown\hspace{.1ex}\raisebox{1ex}{\tiny\mkl{\hspace{-2ex}\scriptstyle 1}}
\\[-1.7ex]
  & x \raisebox{.75ex}{\tiny\mkl{\, \scriptstyle -1}} & \fillbar
\\[-1.6ex]
  &                                                   &          & z
\\[-1.6ex]
  & z \raisebox{-.5ex}{\tiny\mkl{\ \ \scriptstyle 1}} & \fillbar
\\[-1.7ex]
  \blu T_0 = \{y\}\colon
  & & \diagup\hspace{.1ex}\raisebox{-.5ex}{\tiny\mkl{\hspace{-2.5ex}\scriptstyle -1}}
\\[-1.5ex]
  & y\!\!\!
\\[-3.6ex]
\end{array}
  &             &  &             & x &         &
\\[1ex]
  & &&&&\mko{S_0 = \{x,y\}}
\\[-.2ex]
\end{array}
$$
Now count weighted fences as a sum of three terms, one from each fork:
\begin{align*}
   9 D_{{\blu x},{\red yx}} &= 2 + 2 + 1 = 5
\\ 9 D_{{\blu y},{\red yx}} &= -1 -1 + 1 = -1
\\ 9 D_{{\blu z},{\red yx}} &= -1 -1 -2 = -4,
\end{align*}
where $9 = \Delta_{1,{\blu 111}-{\red 131}}$, as calculated before.
This explains the middle column in the top block of ${\blu F_1} \from
{\red F_2}$.
\begin{excise}{%
  The left column, starting from the initial boundary-link $\red
  zy\,\hbox{---}\,zy$, is similar---in fact, the transposition $x
  \leftrightarrow z$ of the computation for~$\red yx$ just completed,
  with different signs to compensate.

  In particular, $S_0 = \{x,y\}$ yields no fences because~$z$ is not a
  stake, and the remaining fences can be drawn with three forks.
  $$%
  \begin{array}{*{3}{@{}c@{}}}
  \\[-2.2ex]
     &&{\red zy \edgehoriz 1 zy}
  \\ &\edgeup{-1}&
  \\ z \makebox[0pt][l]{$\ \not\in S_0$}
  \\[1ex]
     \mko{\qquad\qquad S_0 = \{x,y\}}
  \\[-.2ex]
  \end{array}
  \qquad
  \qquad
  \qquad
  \begin{array}{*{7}{@{}c@{}}}
  \\[-2.2ex]
      &                     &zy&                       &   &&\red zy \edgehoriz 1 zy
  \\  &\edgeup{\!{\sst-}\!1}&  &\edgedown{\!{\sst-}\!1}&   &\edgeup{\!{\sst-}\!1}&
  \\
  \blu
  \begin{array}{@{}r@{\ }c@{\ \,}c@{\,}c@{}}
  \\[-6.3ex]
    & z\!\!\!
  \\[-1.3ex]
    \blu T_0 = \{y\}\colon
    & & \diagdown\hspace{.1ex}\raisebox{1ex}{\tiny\mkl{\hspace{-2ex}\scriptstyle 1}}
  \\[-1.7ex]
    & y \raisebox{.75ex}{\tiny\mkl{\, \scriptstyle -1}} & \fillbar
  \\[-1.6ex]
    &                                                   &          & z
  \\[-1.6ex]
    & z \raisebox{-.5ex}{\tiny\mkl{\ \ \scriptstyle 1}} & \fillbar
  \\[-1.7ex]
    \blu T_0 = \{x\}\colon
    & & \diagup\hspace{.1ex}\raisebox{-.5ex}{\tiny\mkl{\hspace{-2.5ex}\scriptstyle -1}}
  \\[-1.5ex]
    & x\!\!\!
  \\[-3.6ex]
  \end{array}
    &                       &  &                       & z &                     &
  \\[1ex]
    & &&&&\mko{S_0 = \{x,z\}}
  \\[-.2ex]
  \end{array}
  $$
  $$%
  \begin{array}{*{7}{@{}c@{}}}
  \\[-2.2ex]
      &                     &zy&                       &   &&\red zy \edgehoriz 1 zy
  \\  &\edgeup{\!{\sst-}\!1}&  &\edgedown{\!{\sst-}\!1}&   &\edgeup{\!{\sst-}\!1}&
  \\
  \blu
  \begin{array}{@{}r@{\ }c@{\ \,}c@{\,}c@{}}
  \\[-6.3ex]
    & z\!\!\!
  \\[-1.3ex]
    \blu T_0 = \{y\}\colon
    & & \diagdown\hspace{.1ex}\raisebox{1ex}{\tiny\mkl{\hspace{-2ex}\scriptstyle 1}}
  \\[-1.7ex]
    & y \raisebox{.75ex}{\tiny\mkl{\, \scriptstyle -1}} & \fillbar
  \\[-1.6ex]
    &                                                   &          & z
  \\[-1.6ex]
    & z \raisebox{-.5ex}{\tiny\mkl{\ \ \scriptstyle 1}} & \fillbar
  \\[-1.7ex]
    \blu T_0 = \{x\}\colon
    & & \diagup\hspace{.1ex}\raisebox{-.5ex}{\tiny\mkl{\hspace{-2.5ex}\scriptstyle -1}}
  \\[-1.5ex]
    & x\!\!\!
  \\[-3.6ex]
  \end{array}
    &                       &  &                       & z &                     &
  \\[1ex]
    & &&&&\mko{S_0 = \{y,z\}}
  \\[-.2ex]
  \end{array}
  \qquad
  \begin{array}{*{7}{@{}c@{}}}
  \\[-2.2ex]
      &                     &yx&                       &   &&\red zy \edgehoriz 1 zy
  \\  &\edgeup{\!{\sst-}\!1}&  &\edgedown{\!{\sst-}\!1}&   &\edgeup 1&
  \\
  \blu
  \begin{array}{@{}r@{\ }c@{\ \,}c@{\,}c@{}}
  \\[-6.3ex]
    & x\!\!\!
  \\[-1.3ex]
    \blu T_0 = \{z\}\colon
    & & \diagdown\hspace{.1ex}\raisebox{1ex}{\tiny\mkl{\hspace{-2ex}\scriptstyle 1}}
  \\[-1.7ex]
    & z \raisebox{.75ex}{\tiny\mkl{\, \scriptstyle -1}} & \fillbar
  \\[-1.6ex]
    &                                                   &          & x
  \\[-1.6ex]
    & x \raisebox{-.5ex}{\tiny\mkl{\ \ \scriptstyle 1}} & \fillbar
  \\[-1.7ex]
    \blu T_0 = \{y\}\colon
    & & \diagup\hspace{.1ex}\raisebox{-.5ex}{\tiny\mkl{\hspace{-2.5ex}\scriptstyle -1}}
  \\[-1.5ex]
    & y\!\!\!
  \\[-3.6ex]
  \end{array}
    &                       &  &                       & z &                     &
  \\[1ex]
    & &&&&\mko{S_0 = \{y,z\}}
  \\[-.2ex]
  \end{array}
  $$
  Counting weighted fences as before yields
  \begin{align*}
     9 D_{{\blu x},{\red zy}} &= 1 + 1 + 2 = 4
  \\ 9 D_{{\blu y},{\red zy}} &= 1 + 1 - 1 = 1
  \\ 9 D_{{\blu z},{\red zy}} &= -2 - 2 - 1 = -5
  \end{align*}
  to explain the left column of the top block of ${\blu F_1} \from {\red
  F_2}$.
%
}\end{excise}%
\medskip

\begin{excise}{%
  \begin{example}\label{e:hull=square}
  Let $I = \<yz,xz,xy^2,x^2y\>$, whose staircase is depicted here.
  $$%
  \psfrag{x}{$\!x$}
  \psfrag{y}{$y$}
  \psfrag{z}{$z$}
  \includegraphics[height=2.5in]{hull-square}
  $$
  The canonical sylvan resolution of~$I$, notated as per
  Convention~\ref{conv:block-matrix}, is as follows.
  {\foot\vspace{.5ex}
  $$%
  \begin{array}{@{}c@{\!\!\!}}
     \HH_{{\sst\!-\!}1} K^{011} \!\otimes\! \<yz\>\quad
  \\ \oplus
  \\ \HH_{{\sst\!-\!}1} K^{101} \!\otimes\! \<xz\>\quad
  \\ \oplus
  \\ \HH_{{\sst\!-\!}1} K^{120} \!\otimes\! \<xy^2\>\ \,
  \\ \oplus
  \\ \HH_{{\sst\!-\!}1} K^{210} \!\otimes\! \<x^2y\>\ \,
  \end{array}
  \xleftarrow{\!\!\!
  \monomialmatrix
    {\nothing\!\\\nothing\!\\\nothing\!\\\nothing\!}
    {\scriptstyle
     \begin{array}{@{\!}c@{\ }c@{\hspace{-.3ex}}*2{c@{\ }c@{\ }c@{\hspace{-.3ex}}}c@{\ }c@{\!}}
        \blu x\!\!&\blu\!\!y
       &\blu x\!\!&\blu\!y&\blu\!\!z
       &\blu x\!\!&\blu\!y&\blu\!\!z
       &\blu x\!\!&\blu\!\!y
    \\[-.2ex][\,0&0\,]&[\,\nf34&\nf34&0\,]&[\,\nf14&\nf14&0\,]&[\,1&0\,]
    \\[-.2ex][\,0&0\,]&[\,\nf14&\nf14&0\,]&[\,\nf34&\nf34&0\,]&[\,0&1\,]
    \\[-.2ex][\,1&0\,]&[\,\,0\,& 0\, &1\,]&[\,\,0\,& 0\, &0\,]&[\,0&0\,]
    \\[-.2ex][\,0&1\,]&[\,\,0\,& 0\, &0\,]&[\,\,0\,& 0\, &1\,]&[\,0&0\,]
    \end{array}}
    {\\\\\\\\}
  \!\!\!}
  {\blu
  \begin{array}{@{}c@{\,}}
     \HH_{\!0} K^{220} \!\otimes\! \<x^2y^2\>
  \\ \oplus
  \\ \HH_{\!0} K^{121} \!\otimes\! \<xy^2z\>
  \\ \oplus
  \\ \HH_{\!0} K^{211} \!\otimes\! \<x^2yz\>
  \\ \oplus
  \\ \HH_{\!0} K^{111} \!\otimes\! \<xyz\>
  \end{array}
  }
  \xleftarrow{\!\!\!
  \monomialmatrix
    {\blu x\\\blu y\\[.2ex]
     \blu x\\\blu y\\\blu z\\[.2ex]
     \blu x\\\blu y\\\blu z\\[.2ex]
     \blu x\\\blu y}
    {\begin{array}{@{\!}l@{}c@{}r@{\!}}
      \red\ \ zy &\red\ yx &\red xz\ \,
      \\
      \hspace{-.2ex}
      \left[
      \begin{array}{@{}r@{}}
  	  \nfn12
  	\\\nf 12
      \end{array}
      \right.
      &
      \begin{array}{@{}r@{}}
  	  \ \ 0\
  	\\\ \ 0\
      \end{array}
      &
      \left.
      \begin{array}{@{}r@{}}
  	   \nfn12\,
  	\\ \nf 12\,
      \end{array}
      \right]
      \hspace{-.2ex}
      \\[1.5ex]
      \!
      \left[\!
      \begin{array}{@{}r@{}}
  	  \ \ 0
  	\\\ \ 0
  	\\\ \ 0
      \end{array}
      \right.
      &
      \begin{array}{@{}r@{}}
  	   \nf 13
  	\\ \nfn23
  	\\ \nf 13
      \end{array}
      &
      \left.
      \begin{array}{@{}r@{}}
  	   \nfn13
  	\\ \nfn13
  	\\ \nf 23
      \end{array}
      \right]
      \!
      \\[3ex]
      \!
      \left[\!
      \begin{array}{@{}r@{}}
  	   \nf 13
  	\\ \nf 13
  	\\ \nfn23
      \end{array}
      \right.
      &
      \begin{array}{@{}r@{}}
  	   \nf 23
  	\\ \nfn13
  	\\ \nfn13
      \end{array}
      &
      \left.
      \begin{array}{@{}r@{}}
  	   0\,
  	\\ 0\,
  	\\ 0\,
      \end{array}
      \right]
      \!
      \\[3ex]
      \hspace{-.2ex}
      \left[\!
      \begin{array}{@{}r@{}}
  	  \ \ 0
  	\\\ \ 0
      \end{array}
      \right.
      &
      \begin{array}{@{}r@{}}
  	   \nf 12
  	\\ \nfn12
      \end{array}
      &
      \left.
      \begin{array}{@{}r@{}}
  	   0\,
  	\\ 0\,
      \end{array}
      \right]
      \hspace{-.2ex}
      \end{array}}
    {\\\\[.6ex]\\\\\\[.6ex]\\\\\\[.6ex]\\\\}
  \!\!\!}
  {\red \HH_{\!1\hspace{-.1ex}} K^{221} \!\otimes\! \<x^2y^2z\>}
  $$
  }%
  
  \noindent
  Several features are to be noted.  As in
  Example~\ref{e:full-sylvan-res}, the column $\red (1,1,1)^\top$
  generates $\red \HH_1 K^{221}$, this time mapping to the sum $\blu
  (-1,1)^\top \oplus (0,-1,1)^\top \oplus (1,0,-1)^\top
  \oplus\nolinebreak (\nf12,\nfn12)^\top$ in the $\NN^3$-degree~$221$
  component of the middle free module.  This vector of length~$10$ lies
  in the kernel of the ($4$~block)~$\times$~($4$~block) matrix for $F_0
  \from {\blu F_1}$: again, the relevant fractions magically all cancel.
  Even in $F_0 \from {\blu F_1}$ the coefficients are not all~$\pm 1$,
  so the first syzygies do not simply align one whole generator with
  positive sign against another with negative sign.
  
  This example illustrates that lattice paths can cross and overlap
  without harm.
  
  The chain-link fence calculations mostly follow those in
  Example~\ref{e:full-sylvan-res}.  Exception:
  $$%
  \begin{array}{@{}r@{\qquad\quad}c@{\,}c@{\,}c@{\,}c@{\,}c@{\,}c@{}}
  \\[-4ex]
     &    &\qquad\qquad&        &\qquad\qquad&     &\qquad\qquad\\[-2.5ex]
  \\[-2.5ex]\lambda:&
      011    &\fillbar&     111    &\fillbar&   \blu 121
  \\\st_0^{\,\lambda}:&
  \mko{T_{-1} = \{\}}&&\mko{T_0 = \{x\}}&&\mko{{\blu S_0 = \{x\}}}
  \\ &               &&\mko{\phantom{T_0\!}\text{ or }{\{y\}}}
                                        &&\mko{\phantom{S_0\!}\text{ or }{\blu\{y\}}}
  \\ &               &&\mko{S_{-1}=\{\nothing\}}
  \end{array}
  $$
  yields two chain-link fences which, together, contribute $\frac14[1\
  1\ 0]$ to $D^{011,{\blu 121}}$.
  $$%
  \begin{array}{*{7}{@{}c@{}}}
  \\[-2.2ex]
                 &         &x&           &        &         &\blu y \edgehoriz 1 x
  \\             &\edgeup 1& &\edgedown 1&        &\edgeup 1&
  \\ \nothing
     \edgehoriz 1
         \nothing&         & &           &\nothing&         &
  \\[1ex]
     &&\mko{T_0 = \{x\}\quad}&&&&\mko{\quad\blu S_0 = \{x\}}
  \\[-.2ex]
  \end{array}
  \qquad\qquad
  \begin{array}{*{7}{@{}c@{}}}
  \\[-2.2ex]
                 &         &x&           &        &         &\blu y \edgehoriz 1 y
  \\             &\edgeup 1& &\edgedown 1&        &\edgeup 1&
  \\ \nothing
     \edgehoriz 1
         \nothing&         & &           &\nothing&         &
  \\[1ex]
     &&\mko{T_0 = \{x\}\quad}&&&&\mko{\quad\blu S_0 = \{x\}}
  \\[-.2ex]
  \end{array}
  $$
  The other lattice path $\lambda' =
  011\,\hbox{---}\,021\,\hbox{---}\,{\blu 121} \in \Lambda(011,{\blu
  121})$ contributes $\frac12[1\ 1\ 0]$, as in the calculation
  of~$D^{001,{\blu 111}}$ in Example~\ref{e:full-sylvan-res}, to yield a
  total of $D^{011,{\blu 121}} = [\nf34\ \nf34\ 0]$ in the top row of
  $F_0 \from {\blu F_1}$.  Another exception, the calculation of
  $D^{101,{\blu 121}}$, does not follow from
  Example~\ref{e:full-sylvan-res} but is almost exactly the same as
  $\lambda = 011\,\hbox{---}\,021\,\hbox{---}\,{\blu 121}$ just
  computed; the only difference is that $y$ must replace the
  leftmost~$x$ in both fences.
  
  The fences comprising the map ${\blu F_1} \from {\red F_2}$ roughly
  follow the same outline as those in Example~\ref{e:full-sylvan-res}.
  For the reader attempting these calculations, it is useful as a check
  to note that $\Delta_{1,{\blu 111}-{\red 221}} = {\blu 2} \cdot 2
  \cdot {\red 1} = 4$, that each of the two paths in $\Lambda({\blu
  111},{\red 221})$ yields two fences, with each path contributing
  $[\nf14\ \nfn14\ \,0]^\top$ to the middle column.  Finally, given that
  $|{\red 221} - {\blu 121}| = 1$ and that $\Delta_{1,{\blu 121}-{\red
  221}} = {\blu 3} \cdot {\red 1} = 3$, the numbers of chain-link fences
  and their signs are readily deduced from the top middle sylvan matrix
  in~${\blu F_1} \from {\red F_2}$.
  \end{example}
}\end{excise}%
%

\section{Hedge splittings and Moore--Penrose pseudoinverses}\label{s:pseudo}

\begin{defn}\label{d:hedge-splitting}
Fix a linear map $\kk^m \stackrel\del\ffrom \kk^n$.  For any
hedge~$\st$, the \emph{hedge splitting}
$$%
  \del^+_\st: \kk^m \to \kk^n
$$
is defined by its action on the basis $\ol S \cup \del T$:
\begin{enumerate}
\item\label{i:non-stake}%
$\del^+_\st(\sigma) = 0$ for any non-stake~$\sigma \in \ol S$ and

\item\label{i:zeta}%
$\del^+_\st(\del\tau) = \tau$ for any face $\tau \in T$.
\end{enumerate}
\end{defn}

Hedge splittings are related to circuits, shrubs, and hedge rims as
follows.

\begin{prop}\label{p:hedge-splitting}
In the setting of Definition~\ref{d:hedge-splitting}, if $\tau$ is a
standard basis vector then
$$%
  \del^+_\st\del(\tau) = \1 - \zeta_T(\tau),
$$
where $\1$ is the identity on~$\kk^n$ and~$\zeta_T$ is the circuit
projection from Example~\ref{e:unique-vector}.
\end{prop}
\begin{proof}
Since $\tau - \zeta_T(\tau)$ involves only basis vectors in~$T$, it is
fixed by~$\del^+_\st\del$.  Therefore
$$%
  \tau - \zeta_T(\tau)
  =
  \del^+_\st\del\bigl(\tau - \zeta_T(\tau)\bigr)
  = 
 \del^+_\st\del(\tau)
$$
for all standard basis vectors~$\tau \in \kk^n$ because $\zeta_T(\tau)
\in \ker\del$.
\end{proof}

\begin{prop}\label{p:chain-linked}
In the setting of Definition~\ref{d:hedge-splitting}, the shrub of any
stake~$\sigma \in S$ (the chain $s$ from~Lemma~\ref{l:chain-linked})
equals~$\del^+_\st(\sigma)$.
\end{prop}
\begin{proof}
Definition~\ref{d:hedge-splitting} implies that $s =
\del^+_\st\del(s)$.  But $\sigma - \del s$ is a linear combination of
non-stakes by Lemma~\ref{l:chain-linked}, so $\del^+_\st\del(s) =
\del^+_\st(\sigma)$.
\end{proof}

\begin{lemma}\label{l:bS}
Fix a linear map $\kk^m \stackrel\del\ffrom \kk^n$, a stake set~$S$
for~$\del$, and a stake $\sigma \in S$.  The boundary of the shrub
$\del^+_\st(\sigma)$ of~$\sigma$ in the hedge $\st$ depends only
on~$S$ and~$\sigma$, not on the shrubbery~$T$.  More precisely,
$\del\del^+_\st(\sigma) = \beta_S(\sigma)$ is the boundary projection
of~$\sigma$.
\end{lemma}
\begin{proof}
$\del^+_\st(\sigma) = \del^+_\st\bigl(\beta_S(\sigma)\bigr)$ by
Definition~\ref{d:hedge-splitting}.\ref{i:non-stake}, since
$\beta_S(\sigma) \in \sigma + \kk\{\ol S\}$.  But $\beta_S(\sigma)$ is
fixed by $\del\del^+_\st$ because it lies in the image of~$\del$.
\end{proof}

\begin{defn}\label{d:shrub-boundary}
The element $\beta_S(\sigma) \in \del(\kk^n)$ in Lemma~\ref{l:bS} is
the \emph{shrub boundary} of~$\sigma$.
\end{defn}

\begin{defn}\label{d:moore-penrose}
The \emph{Moore--Penrose pseudoinverse} of a linear map $\kk^m
\stackrel\del\ffrom \kk^n$ over a subfield $\kk \subseteq \CC$ of the
complex numbers is the unique homomorphism $\del^+: \kk^m \to
\kk^n$~with
\begin{enumerate}
\item
$\del \del^+ \del = \del$
\item
$\del^+ \del \del^+ = \del^+$
\item\label{i:dd+}%
$(\del \del^+)^* = \del \del^+$
\item\label{i:d+d}%
$(\del^+ \del)^* = \del^+ \del$,
\end{enumerate}
where $^*$ is conjugate transpose.  When $\del$ is the differential of
a CW complex~$K$, the indices are such that $\del$ and~$\del^+$ pass
between two fixed homological stages, so $\del$ would mean
$C_{i-1}\stackrel{\begin{array}{c}\\[-4ex]\scriptstyle
\del_i\!\!\\[-1.5ex]\end{array}}\ffrom C_i$~and then~$\del^+$ would
mean $\del_i^+: C_{i-1} \to C_i$.
\end{defn}

\begin{thm}[{\cite[Theorem~1]{berg1986}}]\label{t:berg}
Fix a linear map $\CC^m \stackrel\del\ffrom \CC^n$ of complex vector
spaces.  The component~$x^+_j\!$ of the Moore--Penrose
solution~$\xx^+\!$ of\/ $\del\xx = \zz$ is expressed as a sum over all
size $r = \mathop{\mathrm{rank}}(\del)$ subsets $S \subseteq
\{1,\dots,m\}$ and $T \subseteq \{1,\dots,n\}$ with~$j \in T$:
$$%
  x^+_j
  =
  (\del^+\zz)_j
  =
  \frac{1\,}{\sum_{S,T}\left\vert\det\del_{S \times T}\right\vert^2}
  \sum_{\substack{|S|\,=\,r\\|T|\,=\,r\\j \in T}}
    \bigl\vert\det\ol\del_{S \times T}
    \det\bigl((\del_{S \times T})_j[\zz_S]\bigr)\bigr\vert,
$$
where the bar over $\del$ denotes complex conjugation and
\begin{itemize}
\item%
$\del_{S \times T}$ restricts~$\del$ to its submatrix with rows and
columns indexed by~$S$ and~$T$;
\item%
$\zz_S$ restricts the column vector~$\zz$ to its entries indexed
by~$S$; and
\item%
$(\del_{S \times T})_j[\zz_S]$ replaces column~$j$ of\/~$\del_{S
\times T}$ with~$\zz_S$.
\end{itemize}
\end{thm}

\begin{remark}\label{r:berg}
To be faithful to \cite{berg1986}, the sums in Theorem~\ref{t:berg}
are taken over arbitrary size~$r$ subsets.  This contrasts with most
of the summations in this paper, which restrict to summands where $S$
is a stake set and $T$ is a shrubbery.  However, Berg's sums might as
well be over hedges $ST$, as the determinant in every remaining term
vanishes; this is a simple but key point in the proof of the following
consequence.
\end{remark}

\begin{cor}[Hedge Formula]\label{c:hedge}
Fix a subfield\/~$\kk \subseteq \CC$ of the complex numbers.  The
Moore--Pen\-rose pseudoinverse of a linear map \smash{$\kk^m
\stackrel\del\ffrom \kk^n$} is a sum over~hedges~$ST$~for~$\del$:
$$%
  \del^+
  =
  \frac{1\,}{\sum_\st\left\vert\det\del_{S \times T}\right\vert^2}
  \sum_\st \left\vert\det\del_{S \times T}\right\vert^2\del^+_\st\,.
$$
\end{cor}
\begin{proof}
First assume $\kk = \CC$.  By Remark~\ref{r:berg}, the summations in
Berg's formula (Theorem~\ref{t:berg}) can be taken over all
hedges~$ST$ for~$\del$.  Fix a hedge~$\st$.

If $\zz \in \oS$, then $\zz_S = 0$, so the summand contributed
by~$\st$ in Berg's formula vanishes because $(\del_{S \times
T})_j[\zz_S] = 0$, which agrees with
Definition~\ref{d:hedge-splitting}.\ref{i:non-stake}.  For
Definition~\ref{d:hedge-splitting}.\ref{i:zeta}, let $\zz = \del\tau$
for some $\tau \in T$, so $\del_\st^+ \del\tau = \tau$.  If $\tau \in
S$, then in the sum over~$j \in T$ of Berg's formula, with $\st$
fixed, the second determinant is $\det\del_{S \times T}$ when $\tau$
replaces itself and otherwise yields a matrix with repeated columns
and hence vanishing determinant.  If $\tau \not\in S$ then the
replacement always yields repeated columns and hence determinant~$0$.

If~$\zz \in S$ then the coefficient of $\tau \not\in T$ in
$\del_\st^+\zz$ vanishes for all~$\zz$, because $\del_\st^+ \sigma =
0$ for all \smash{$\sigma \in \oS$} and $\del_\st^+\del\tau = \tau$
for all $\tau \in T$.  Therefore the sum can be taken over all hedges
$\st$, not just those that include a certain vector in the
shrubbery~$T$.

For subfields of~$\CC$, the formula for $\del^+$ is defined and has
the claimed properties after tensoring with~$\CC$.  By flatness it
must have had those properties before tensoring.
\end{proof}

\begin{prop}\label{p:hedge}
If\/ $\ZZ^m \stackrel\del\ffrom \ZZ^n$, then over any field\/~$\kk$ in
which the denominator $\sum_\st\det(\del_{S \times T})^2$
and the order of the torsion subgroup of\/~$\ZZ^m/\del(\ZZ^n)$ are
invertible, the summation formula for~$\del^+$ defines a splitting
$\del^+_\kk$ of the
map $\kk^m \stackrel{\smash{\del_\kk}}\ffrom \kk^n$ induced
by~$\mbox{}\otimes \kk$.
\end{prop}
\begin{proof}
Over the rationals~$\QQ$, the Moore--Penrose pseudoinverse of~$\del$
is an orthogonal projection $\pi_B: \QQ^m \onto B$ followed by an
isomorphism $B \simto K^\perp$ of the image~$B = \del_\QQ(\QQ^m)$ with
the orthogonal complement in~$\QQ^n$ of the kernel $K = \ker(B \otno
\QQ^n)$.

View the inclusion $B^\ZZ = \del(\ZZ^n) \subseteq \ZZ^m$ as taking
place inside of~$\QQ^m$, which contains~$B$ as well as~$B^\ZZ$
and~$\ZZ^m$.  Let $R$ be the localization of~$\ZZ$ by inverting the
denominator and the torsion order.  Then $\pi_B(R^m)$ lies in $B \cap
R^m$ because the denominator is inverted.  But because the torsion
order is inverted, $B \cap R^m = B^\ZZ \otimes R$, a direct
summand~$B^R$ of~$R^m$.  Since $\pi_B$ fixes~$B$ and hence~$B^R$, it
follows that $\pi_B(R^m) = B^R$.  This surjectivity of $\pi_B^R: R^m
\onto B^R$ persists modulo any prime of~$R$, by right-exactness of
tensor~products, and subsequently under any extension of scalars.

The isomorphism $B \simto K^\perp \subseteq \QQ^n$ over~$\QQ$
restricts to an isomorphism of~$B^R$ with a subgroup of~$R^n$ because
it is split by~$\del_\QQ$.  This splitting is preserved by arbitrary
tensor products, including quotients modulo primes and subsequent
extension of scalars.
\end{proof}

\begin{cor}[Projection Hedge Formula]\label{c:projection-hedge}
Fix a subfield\/~$\kk \subseteq \CC$ and a subspace $A \subseteq
\kk^\ell$.  Using $\alpha_U:\nolinebreak \kk^\ell \onto A$ from
Lemma~\ref{l:alpha}, the orthogonal projection $\pi_A:\kk^\ell \onto
A$~is
$$%
  \pi_A
  =
  \frac{1\,}{\sum_U\left\vert\det\eta_U\right\vert^2}
  \sum_U \left\vert\det\eta_U\right\vert^2 \alpha_U,
$$
where the sums are equivalently over all
\begin{enumerate}
\item%
shrubberies $T = U$ for $\kk^\ell/A \!\otno\! \kk^\ell$, in which case
$\eta_U = \del_T$ (Definition~\ref{d:hedge}.\ref{i:shrubbery'}),~or
\item%
stake sets $S = \oU$ for $\kk^\ell \!\otni\! A$, in which case $\eta_U
= \del_S$ (\mbox{Definition}~\ref{d:hedge}.\ref{i:stake'}),
\end{enumerate}
and the determinants are calculated using $U$ along with any basis
of\/ $\kk^\ell/A$ or~$A$.
\end{cor}
\begin{proof}
The formula does not depend on the basis of~$A$ because the
change-of-basis determinant multiplies the numerator and the
denominator equally.

The equivalence of the shrubbery and stake set formulations is
plausible because a shrubbery~$T$ for the surjection $\kk^\ell/A \otno
\kk^\ell$ is defined by the condition $\kk^\ell = A \oplus \kk^T$
while a stake set~$S$ for the injection $\kk^\ell \otni A$ is defined
by the condition \smash{$\kk^\ell = \kk^\oS \oplus A$}.  Thus each
shrubbery~$T$ uniquely corresponds to a stake set~$\oT$.  While it it
not necessarily the case that $\det\del_T = \det \del_S$ for $U = T =
\oS$, the ratios between these determinants are constant as $U$
varies, because $\bigwedge^d \kk^S \otimes \bigwedge^{\ell-d} \kk^\oS
= \bigwedge^\ell \kk^\ell = \bigwedge^d A \otimes \bigwedge^{\ell-d}
\kk^\ell/A$, where $d = |S|$, while the maps $\bigwedge^d \kk^S \from
\bigwedge^d A$ and $\bigwedge^{\ell-d} \kk^\oS \to \bigwedge^{\ell-d}
\kk^\ell/A$ go in opposite~directions.

To prove the stake set formulation, set $\ell = m$ and choose any
linear transformation \smash{$\kk^m \stackrel\del\ffrom \kk^n$} with
image $A = B$.  Multiply the Hedge Formula (Corollary~\ref{c:hedge})
on the left by~$\del$.  This yields left-hand side and $\pi_B = \del
\del^+$ and, by Lemma~\ref{l:bS}, right-hand side
\begin{align*}
\frac{1\,}{\sum_\st\left\vert\det\del_{S \times T}\right\vert^2}
\sum_\st \left\vert\det\del_{S \times T}\right\vert^2\del\del^+_\st
   &=  \frac{1\,}{\sum_\st\left\vert\det\del_{S \times T}\right\vert^2}
       \sum_\st \left\vert\det\del_{S \times T}\right\vert^2\beta_S.
\end{align*}
Using any basis of~$A$ to compute determinants $\del\del_T$
and~$\det\del_S$, the fact that $\del_{S \times T} = \del_S \circ
\del_T$ implies that $\det(\del_{S \times T})^2 =
\det(\del_S)^2\det(\del_T)^2$.  Therefore a~factor of
$\sum_T\left\vert\det\del_T\right\vert^2$ pulls out of the numerator
as well as the denominator, leaving the desired sum.
\end{proof}

\begin{prop}\label{p:projection-hedge}
Fix an integral structure $B^\ZZ \to B$ (Definition~\ref{d:dets}) on a
surjection, injection, or based linear map~$\del$ as in
Definition~\ref{d:hedge}.  If the order of the torsion subgroup
of\/~$\ZZ^\ell/B^\ZZ$ and denominator $\sum_U\det(\del_U)^2$ are
invertible in~$\kk$, then the formula for~$\pi_A$ in
Corollary~\ref{c:projection-hedge} defines a surjection $\pi_A:
\kk^\ell \onto A$ that splits the inclusion $\kk^\ell \otni A$~when
\begin{itemize}
\item%
$\ell = n$ and \smash{$A = \ker(B
\stackrel{\begin{array}{c}\scriptstyle\:\del\\[-1ex]\end{array}}\fillotno
\kk^n)$}, the sum is interpreted as being over shrubberies $T = U$ (so
$\del_U$ is the isomorphism~$\del_T$ from Definition~\ref{d:hedge}),
and $\alpha_U = \zeta_T$; or
\item%
$\ell = m$ and $A = B$, the sum is interpreted as being over stake
sets $S = \oU$ (so $\del_U$ is the isomorphism~$\del_S$ from
Definition~\ref{d:hedge}), and $\alpha_U = \beta_S$.
\end{itemize}
\end{prop}
\begin{proof}
Argue as in the second paragraph of the proof of
Proposition~\ref{p:hedge}, mutatis mutandis.  All the integral
structure does is to make the determinants well defined.
\end{proof}

\section{Koszul bicomplexes}\label{s:koszul}

\begin{conv}[Koszul complex notation]\label{conv:koszul}
Let $V$ be a vector space over~$\kk$ of dimension~$n$ that is
$\NN^n$-graded to have one basis vector $z_1,\dots,z_n$ in each of the
degrees~$\ee_1,\dots,\ee_n$ of the variables~$x_1,\dots,x_n$ of the
polynomial ring~$\kk[\xx]$.
The Koszul complexes on the variables~$\xx$ and on variables $\yy =
y_1,\dots,y_n$ are denoted by
$$%
\textstyle
  \KK^\xx_\spot
  =
  \bigwedge^{\!\spot\!} V \otimes \kk[\xx]
\quad\text{and}\quad
  \KK^\yy_\spot
  =
  \bigwedge^{\!\spot\!} V \otimes \kk[\yy]
$$
with their usual $\NN^n$-graded differentials.  Thus, for example, the
degree~$\bb$ differential
$$%
  (\KK^\yy_{i-1})_\bb
  =
  \bigoplus_{|\sigma| = i-1} \zz^\sigma \otimes \yy^{\bb-\sigma}
  \ \ \from\ \
  \bigoplus_{|\sigma| = i} \zz^\sigma \otimes \yy^{\bb-\sigma}
  =
  (\KK^\yy_i)_\bb
$$
is induced by the $\NN^n$-graded $\kk[\yy]$-linear map $\kk[\yy] \from
V \otimes \kk[\yy]$ that sends $y_j \mapsfrom z_j \otimes 1$:
$$%
  \sum_{j \in \sigma} \pm\,\zz^{\sigma-\ee_j} \otimes \yy^{\bb+\ee_j-\sigma}
  \ \ \mapsfrom\ \
  \zz^\sigma \otimes \yy^{\bb-\sigma}.
$$
The symbol $\zz^\sigma$ denotes the exterior product of the basis
vectors of~$V$ indexed by the simplex $\sigma \subseteq
\{1,\dots,n\}$, which is also identified with its characteristic
vector~in~$\{0,1\}^n$.  The monomial~$\zz^\sigma$ is a
$\kk[\xx]$-basis for a rank~$1$ free $\NN^n$-graded
summand~$\KK^\xx_\sigma$ of~$\KK^\xx_{|\sigma|}$, and similarly
for~$\KK^\yy_\sigma$.  For any $\kk[\xx]$-module~$M$, write $M^\yy$
for the corresponding $\kk[\yy]$-module, and let $M^{-\yy}$ be the
same module but where each variable $y_j$ acts by~$-y_j$.
\end{conv}

\begin{defn}\label{d:KK}
Using equality signs for natural isomorphisms of modules over the ring
$\kk[\xx,\yy] = \kk[\xx] \otimes \kk[\yy]$, the \emph{Koszul
bicomplex} $\KK_\sspot$ is equivalently
$$%
\begin{array}{rcccl}
  \kk[\xx] \otimes \KK_\spot^\yy 
& \!\!=\!\!
& \kk[\xx] \otimes \bigwedge^{\!\spot\!} V \otimes \kk[\yy]
& \!\!=\!\!
& \KK_\spot^\xx \otimes \kk[\yy]
\\
& 
& \verteq
& 
& 
\\
& 
& \KK^{\xx+\yy}_\spot
& 
& 
\\
\end{array}
$$
with \raisebox{2.1ex}{\parbox[t]{.9\linewidth}{
\begin{itemize}
\item%
\emph{horizontal differential} induced by $\KK^\xx_\spot$
\item%
\emph{vertical differential} induced by $\KK^\yy_\spot$ and
\item%
\emph{total differential} induced by $\KK^{\xx+\yy}_\spot$
\end{itemize}}}\\
where $\xx + \yy = x_1 + y_1, \dots, x_n + y_n$ lies in
$\kk[\xx,\yy]$.  For any $\kk[\xx]$-module~$M$, write
$$%
\KK_\sspot(M) = \KK_\sspot \otimes_{\kk[\yy]} M^{-\yy}
$$
for the \emph{Koszul bicomplex of~$M$}.
\end{defn}


\begin{prop}\label{p:tot-KK}
The total complex $\KK^{\xx+\yy}_\spot(M)$ is a $\kk[\xx]$-free
resolution of any $\kk[\xx]$-module~$M$ as a module
over~$\kk[\xx,\yy]$ on which $y_j$ acts as~$-x_j$ for $j = 1,\dots,n$.
\end{prop}
\begin{proof}
The total complex is free over~$\kk[\xx]$ because
$\KK^{\xx+\yy}_\spot(M) = \KK^\xx_\spot \otimes \kk[\yy]
\otimes_{\kk[\yy]} M^{-\yy}$.  That this complex can also be expressed
as $\KK_\sspot \otimes_{\kk[\xx,\yy]} \kk[\xx,\yy] \otimes_{\kk[\yy]}
M^{-\yy}$ means that it is the Koszul complex for the sequence
$\xx+\yy$ acting on $\kk[\xx,\yy] \otimes_{\kk[\yy]} M^{-\yy}$.  This
sequence is regular, so $\KK^{\xx+\yy}_\spot(M)$ is acyclic.  Its
nonzero homology is naturally the $\kk[\yy]$-module $M^{-\yy}$ with an
action of~$\kk[\xx]$ in which the variable $x_j$ acts on~$M^{-\yy}$
the way $-y_j$ acts.  As a $\kk[\xx]$-module, this is just~$M$.
\end{proof}

\begin{remark}\label{r:total}
If $I \subseteq \kk[\xx]$ is any ideal and~$K$ is any
$\kk[\xx,\yy]$-module, such as $\KK^{\xx+\yy}_\spot$, then
$$%
  K \otimes_{\kk[\yy]} I^{-\yy}
  \cong
  K \otimes_{\kk[\xx,\yy]} \kk[\xx,\yy] \otimes_{\kk[\yy]} I^{-\yy}
  \cong
  K \otimes_{\kk[\xx,\yy]} I^{-\yy}\kk[\xx,\yy],
$$
where the second isomorphism is by flatness of $\kk[\xx,\yy]$
over~$\kk[\yy]$.
%
Thus $\KK_\sspot(I)$ is the ordinary Koszul complex of the ideal
$I^{-\yy} \kk[\xx,\yy]$.  Also note that when $I$ is graded, $I^{-\yy}
= I^\yy$; indeed, $M^{-\yy} = M^\yy$ as $\kk$-vector spaces for any
$\ZZ$-graded $\kk[\xx]$-module~$M$.
\end{remark}

\begin{lemma}\label{l:KK}
When $M$ is a $\ZZ$-graded $\kk[\xx]$-module, the double indexing
$$%
  \KK_{pq}(M) = \kk[\xx] \otimes \KK^\yy_{p+q} \otimes M_{-q}
$$
makes the Koszul bicomplex of~$M$ into a fourth-quadrant bicomplex
$\KK_\sspot(M)$ of modules over~$\kk[\xx]$ concentrated in a strip
between the diagonals $p + q = 0$ and $p + q = n$.
\end{lemma}
\begin{proof}
The content of the claim is that $(\KK^\yy_{p+q} \otimes_{\kk[\yy]}
M^{-\yy})_p = \KK^\yy_{p+q} \otimes M^{-\yy}_{-q} = \KK^\yy_{p+q}
\otimes M_{-q}$ (see Remark~\ref{r:total}) after which the lemma is
proved by tensoring with~$\kk[\xx]$.
\end{proof}

\begin{thm}\label{t:KK-spectral-seq}
The Koszul bicomplex $\KK_\sspot(M)$ for any graded
$\kk[\xx]$-module~$M$ has vertical-then-horizontal spectral sequence
$$%
  \tor_{p+q}(\kk,M)_p \otimes \kk[\xx]
  \implies
  H_{p+q}\KK^{\xx+\yy}_\spot(I).
$$
\end{thm}
\begin{proof}
The vertical homology of~$\KK_\sspot(M)$ at the location $pq$ is
$\kk[\xx] \otimes \tor_{p+q}(\kk,M)_p$ by Lemma~\ref{l:KK}
(particularly its indexing).  The rest is by Definition~\ref{d:KK}.
\end{proof}

\begin{cor}\label{c:KK-spectral-seq}
For a monomial ideal~$I$, the vertical-then-horizontal spectral
sequence of its Koszul bicomplex~$\KK_\sspot(I)$ is
$$%
  \bigoplus_{|\aa| = p} H_{p + q - 1} K^\aa I \otimes \kk[\xx]
  \implies
  H_{p+q}\KK^{\xx+\yy}_\spot(I).
$$
\end{cor}
\begin{proof}
Apply Hochster's formula (Theorem~\ref{t:Kb}) to
Theorem~\ref{t:KK-spectral-seq}.
\end{proof}

\begin{remark}\label{r:KK-spectral-seq}
The spectral sequence in Corollary~\ref{c:KK-spectral-seq} produces
arrows as in Problem~\ref{prob:kaplansky}.  However, by the nature of
spectral sequences, these arrows only represent homomorphisms between
subquotients of the relevant free modules and therefore cannot
directly be differentials in a free resolution of~$I$.  The next
section is the remedy.
\end{remark}

\section{Minimal free resolutions from Wall complexes}\label{s:wall}

\noindent
The default coefficient ring in this section is an arbitrary ring~$R$.

This section is mainly a summary of constructions and main results on
Wall complexes from \cite{eagon1990}, without repeating proofs.  The
goal is to deduce that derived Wall complexes of Koszul bicomplexes
are minimal free resolutions (Corollary~\ref{c:wall-koszul}).

  The key point is that the Wall complex
should resolve the given ideal~$I$, and not some associated graded
module~$\gr I$.

\begin{defn}\label{d:wall}
Fix a ring~$R$ and a doubly indexed array $W_\sspot$ of $R$-modules
with maps $\omega_j: W_{pq} \to W_{p-j,q+j-1}$ for $j \in \NN$ (the
index~$pq$ on~$\omega_j$ is suppressed).  Assume that for each element
$w \in W_{pq}$, only finitely many images $\omega_j(w)$ are nonzero.
Set
$$%
  W_i = \bigoplus_{p+q = i} W_{pq}
  \quad\text{and}\quad
  D_i = \sum_{j=0}^\infty \omega_j: W_i \to W_{i-1}.
$$
These data constitute a \emph{Wall complex} if $D^2 = 0$, and
$W_\spot$ with the differential~$D$ is the \emph{total complex}
of~$W_\sspot$.
\end{defn}

\begin{defn}\label{d:vertically-split}
Fix a bicomplex $C_\sspot$ of $R$-modules with vertical
differential $d = d_0$ and horizontal differential~$d_1$.  A
\emph{vertical splitting} of~$C_\sspot$ consists of a differential
$$%
  d^+ = d^+_{pq}: C_{pq} \to C_{p,q+1}
$$
with $d d^+ d = d$ and $d^+ d d^+ = d^+$.  The condition of being a
differential means~$d^+d^+ =\nolinebreak 0$.
\end{defn}

Thus $d^+$ is a vertical cohomological differential, going up columns
opposite to the homological vertical differential~$d$.  The following
is elementary \cite[Proposition~1.1]{eagon1990}.

\begin{prop}\label{p:splitting}
A vertical splitting of~$C_\sspot$ is equivalent to a direct sum
decomposition
$$%
  C_{pq}
  =
  B'_{p,q-1} \oplus \underbrace{H_{pq} \oplus B_{pq}}_{\textstyle Z_{pq}}
$$
in which, for all indices $p$ and~$q$,
\begin{itemize}
\item%
$H_{pq} \oplus B_{pq} = Z_{pq} = \ker d_{pq}$ and 
\item%
$\ol d_{p,q-1}: B'_{p,q-1} \simto B_{p,q-1} = \image d_{pq}$, where
$\ol d_{p,q-1}$ is the restriction of~$d_{pq}$ to~$B'_{p,q-1}$.
\end{itemize}
More precisely, a vertical splitting is constructed from this direct
sum decomposition by
$$%
  d^+_{pq} = \iota_{pq} \circ \ol d_{pq}^{\,-1} \circ \pi_{pq},
$$
where $\pi_{pq}: C_{pq} \onto B_{pq}$ projects to the summand~$B_{pq}$
and $\iota_{pq}: B'_{pq} \into C_{p,q+1}$ is~inclusion.
\end{prop}

The homomorphisms whose composites define Wall complexes from
bicomplexes are elementary to isolate.

\begin{proof}[\bf Lemma~\ref{l:dd+}]\it\refstepcounter{thm}\label{l:dd+}
Fix a vertical splitting of~$C_\sspot$, with notation as in
Proposition~\ref{p:splitting}.
\begin{enumerate}
\item\label{i:P}%
The homology projection $C_\sspot \onto H_\sspot$~is
$P = \1 - d d^+ - d^+ d$.
\item%
The composite of the upward and leftward differentials induces
homomorphisms
$$%
\begin{array}{c@{}c@{}c}
  \makebox[0pt][l]{$C_{p-1,q+1}$}\qquad\ \
\\[.5ex] &\nwarrow\raisebox{.6ex}{\makebox[0pt][l]{$\scriptstyle\!\!\!\!d^+ d_1$}}
\\            && C_{pq}
\end{array}
$$
\end{enumerate}
Together, these homomorphisms induce morphisms $\omega_j: H_{pq} \to
H_{p-j,q+j-1}$ for $j \geq 1$~via
\begin{align*}
  \omega_j &= P (d_1 d^+)^{j-1} d_1.\qedhere
\end{align*}
\end{proof}


\begin{defn}\label{d:derived-wall}
The \emph{derived Wall complex} of a bicomplex~$C_\sspot$ with
vertical differential~$d$ split by~$d^+$ is $H_\sspot$ with the
differentials $\omega_0 = 0$ and $\omega_i$ from Lemma~\ref{l:dd+}
for~$i \geq 1$.
\end{defn}

\begin{remark}\label{r:natural-wall}
Let $C_\sspot$ be a vertically split bicomplex.  The derived Wall
complex selects a split submodule $H_{pq} \subseteq Z_{pq}$ inside the
vertical cycles of~$C_{pq}$ that maps isomorphically to the vertical
homology---naturally defined as the quotient of these same cycles
modulo the boundary submodule---which we denote by $\HH_{pq} =
Z_{pq}/B_{pq}$ so as not to confuse it with the submodule~$H_{pq}$.
The Wall differential $\omega_j = P (d_1 d^+)^{j-1} d_1$ assumes that
its input is an element of the split homology submodule~$H_{pq}$.  In
applications such as to Problem~\ref{prob:kaplansky}, where one wishes
to specify homomorphisms
$$%
  \wt\omega_j: \HH_{pq} \to \HH_{p-j,q+j-1}
$$
on natural homology, the input should be a homology class---specified
as a cycle that is well defined only up to adding a boundary element,
rather than specified as an element of the split
submodule~$H_{pq}$---but at a cost: $\wt\omega_j$ must first project
$Z_{pq}$ to~$H_{pq}$ to ensure that the Wall differential acts
indistinguishably on different cycles representing the same homology
class.  This projection is
$$%
  \1 - d d^+: Z_{pq} \to H_{pq}.
$$
On the other hand, $\wt\omega_j$ need not be forced to produce output
that lies in the split submodule~$H_{p-j,q+j-1}$; it need only produce
a cycle in~$Z_{p-j,q+j-1}$, since the output of~$\wt\omega_j$ is to be
understood modulo~$B_{pq}$.  That means $\wt\omega_j$ can use the
simpler~projection
$$%
  \1 - d^+ d: C_{pq} \to Z_{pq}
$$
from chains to cycles instead of the split homology projection $P = \1
- d d^+ - d^+ d$ from Lemma~\ref{l:dd+}.\ref{i:P}.  In total, then,
the projection $dd^+$ moves from the left end of the expression
defining~$\omega_j = (\1 - d d^+ - d^+ d) (d_1 d^+)^{j-1} d_1$ to the
right end of the~expression
$$%
  \wt\omega_j
  =
  (\1 - d^+ d) (d_1 d^+)^{j-1} d_1 (\1 - d d^+)
  :
  Z_{pq} \to Z_{p-j,q+j-1},
$$
which defines a differential---the same differential as $\omega_j$
defines---because
\begin{align*}
(\1 - dd^+)(\1 - d^+d)
   &= \1 - dd^+ - d^+d + dd^+d^+d
\\ &= \1 - dd^+ - d^+d
\\[-4.05ex]
\end{align*}
occurs between the factors of $(d_1 d^+)^{j-1} d_1$ and $(d_1
d^+)^{j'-1} d_1$ in the square of the Wall differential either way.
\end{remark}

It is these differentials~$\wt\omega_j$, rather than $\omega_j$ from
Lemma~\ref{l:dd+}, that solve Problem~\ref{prob:kaplansky} and give
rise to the combinatorics in Section~\ref{s:sylvan}.
We therefore record this shift from split homology~$H_\sspot$ to
natural homology~$\HH_\sspot$ formally, the proof being in
Remark~\ref{r:natural-wall}.

\begin{defn}\label{d:natural-wall}
The \emph{natural Wall complex} of a bicomplex~$C_\sspot$ with
vertical differential~$d$ split by~$d^+$ is $\HH_\sspot$ with the
differentials $\wt\omega_0 = 0$ and $\wt\omega_i$ from
Remark~\ref{r:natural-wall} for~$i \geq 1$.
\end{defn}

\begin{prop}[{\cite[Theorem~1.2]{eagon1990}}]\label{p:wall}
The derived Wall complex~$H_\sspot$ of a vertically split
bicomplex~$C_\sspot$ is a Wall complex as long as the local finiteness
of~$\omega_\spot$ is satisfied.  The total complex of~$H_\sspot$ has a
filtration by taking successively more columns, starting from the
left.  The spectral sequence $HE^\spot$ for this filtration
of~$H_\sspot$ is the same as the vertical-then-horizontal spectral
sequence $E^\spot$ of~$C_\sspot$, in the sense that $HE^r_{pq} \cong
E^r_{pq}$~for~$r \geq 1$.
\end{prop}

\begin{prop}\label{p:natural-wall}
Using the natural Wall complex~$\HH_\sspot$ in place of~$H_\sspot$ and
$\wt\omega_\spot$ in place of~$\omega_\spot$ in
Proposition~\ref{p:wall}, its conclusions hold verbatim.\qed
\end{prop}

\begin{cor}\label{c:wall-koszul}
Fix an arbitrary standard\/ $\ZZ$-graded module~$M$ over~$\kk[\xx]$.
The total complex $W_\spot(M)$ of the derived or natural Wall
complex~$W_\sspot(M)$ for any vertical splitting of the Koszul
bicomplex~$\KK_\sspot(M)$ is a minimal free resolution of~$M$.
\end{cor}
\begin{proof}
Remark~\ref{r:natural-wall} explains why the derived and natural Wall
complexes have the same homology, so it is only necessary to prove the
derived Wall case, for which
\cite[Lemma~3.5]{eisenbud-floystad-schreyer2003} provides a brief
proof in the generality of abelian categories.
\end{proof}

\section{Monomial resolutions from splittings}\label{s:resolutions}

The first result in this section accomplishes the non-combinatorial
part of the proof of Theorem~\ref{t:sylvan}.  It is separated out from
the rest of the proof because it applies in much more generality than
the canonical sylvan setting, which is the choice to use the
Moore--Penrose pseudoinverses as the Koszul simplicial splittings.
Notation for saturated decreasing lattice paths is as in
Definition~\ref{d:hedgerow}: $\lambda \in \Lambda(\aa,\bb)$ is
described as $\bb = \bb_0, \bb_1, \dots, \bb_{\ell-1}, \bb_\ell = \aa$
or its successive differences $(\lambda_1,\dots,\lambda_\ell)$ with
$\lambda_j = \bb_{j-1} - \bb_j$.

\begin{thm}\label{t:choice-of-splitting}
Fix a monomial ideal~$I$.  Any splittings~$\del^{\bb+}$ of the
differentials~$\del^\bb$ of the Koszul simplicial complexes $K^\bb I$
for $\bb \in \NN^n$ that are themselves differentials satisfying
\begin{enumerate}
\item%
$\del^\bb \del^{\bb+} \del^\bb = \del^\bb$ and
\item%
$\del^{\bb+} \del^\bb \del^{\bb+} = \del^{\bb+}$
\end{enumerate}
yield a minimal free resolution of~$I$ whose differential from
homological stage $i+1$ to stage~$i$ has its component $\HH_{i-1}
K^\aa I \otimes\kk[\xx] (-\aa) \from \HH_i K^\bb I \otimes \kk[\xx]
(-\bb)$ induced by~the~map
$$%
  \HH_{i-1} K^\aa I
  \stackrel{\ D}\ffrom
  \HH_i K^\bb I
$$
in $\NN^n$-degree~$\bb$
that acts on any $i$-cycle in~$\wt Z_i K^\bb I$ via
$$%
  D\
  =
  \!\sum_{\lambda \in \Lambda(\aa,\bb)}
  (I^\aa - \del^{\aa+}_i \del^\aa_i)
  d_1^{\lambda_\ell}
  \Bigl(
  \prod_{j=1}^{\ell-1}
  \del^{\bb_j+}_i
  d_1^{\lambda_j}
  \Bigr)
  (I^\bb - \del^\bb_{i+1} \del^{\bb+}_{i+1}),
$$
where $d_1^{\lambda_j\!}$ takes $\tau \subseteq \{1,\dots,n\}$ to~$0$
if
$\lambda_j \not\in \tau$ and to $(-1)^{\tau \minus \lambda_j \subset
\tau}\, \tau\!\minus\!\lambda_j$ if $\lambda_j \in \tau$.
\end{thm}

\begin{remark}\label{r:visualize-maps}
To visualize the formula for~$D$, read the ``$\del$'' and ``$d$'' maps
in the following diagram from right to left, ignoring height, but note
that the products of the rightmost up-down and leftmost down-up maps
must be subtracted from the identity.
$$%
\begin{array}{c@{\ }c@{}c@{\ \ \,}c@{\ \,}c@{}c@{\!\!}c@{}}
  \wt C_{i-1} K^\aa I
& \xleftarrow{\textstyle\,d_1^{\lambda_\ell\,}}\qquad\qquad\quad
\\[1ex]
  \del^{\aa+}_i\longuparrow\!\longdownarrow\del^\aa_i\ \ 
& \qquad\longuparrow\del^{\bb_{\ell-1}+}_i
\\
& \wt C_{i-1} K^{\bb_{\ell-1}} I
& \xleftarrow{\textstyle d_1^{\lambda_{\ell-1}\!\!\!}}
\\[3ex]
& 
& \ddots
& \qquad\longuparrow\del^{\bb_2+}_i
\\
& 
& 
& \wt C_{i-1} K^{\bb_2} I
& \xleftarrow{\textstyle\,d_1^{\lambda_2}\,}\qquad\qquad
\\[1ex]
& 
& 
& 
& \qquad\longuparrow\del^{\bb_1+}_i
& 
& \del^\bb_{i+1}\longdownarrow\longuparrow\del^{\bb+}_{i+1}
\\
& 
& 
& 
& \wt C_{i-1} K^{\bb_1} I
& \xleftarrow{\textstyle\,d_1^{\lambda_1}\,}
& \wt C_i K^\bb I
\end{array}
$$
This diagram, when rotated counterclockwise by~$\pi/4$, is a
zoomed-in, labeled version of the chain-link fence in
Definition~\ref{d:fence}.
\end{remark}

The proof requires notation in which to make concerete computations.
It is nothing more than Definition~\ref{d:KK} expressed explicitly in
coordinates for a monomial ideal~$I$.

\begin{lemma}\label{l:basis}
$\KK_\sspot(I) \cong \bigwedge^{\!\spot\!} V \otimes I^\yy \otimes
\kk[\xx]$ has a $\kk$-linear basis $\zz^\tau \otimes \yy^\bb \otimes
\xx^\aa$ for
\begin{itemize}
\item%
$\zz^\tau \in \bigwedge^{\!|\tau|\!} V$,
\item%
$\yy^\bb \in I^\yy$, and
\item%
$\xx^\aa \in \kk[\xx]$.
\end{itemize}
The $\NN^n$-degree of $\zz^\tau \otimes \yy^\bb \otimes \xx^\aa$ is
$\tau + \bb + \aa$.  The differentials of~$\KK_\sspot$ in this basis
are
$$%
\begin{array}{ccc}
  \sum_{k \in \tau} (-1)^{\tau \minus k \subset \tau}\zz^{\tau-\ee_k}
                 \otimes \yy^\bb
                 \otimes \xx^{\aa+\ee_k}
& \stackrel{\textstyle\,d_1}\mapsfrom
& \zz^\tau \otimes \yy^\bb \otimes \xx^\aa
\\[-1ex]
&
& \downmapsto\,\raisebox{-1.5ex}{$d$}
\\
&
& \makebox[0pt][c]{$
  \sum_{k \in \tau} (-1)^{\tau \minus k \subset \tau}\zz^{\tau-\ee_k}
                 \otimes \yy^{\bb+\ee_k}
                 \otimes \xx^\aa$.}
\end{array}
$$ %
\end{lemma}

\begin{proof}[Proof of Theorem~\ref{t:choice-of-splitting}]
The vertical differentials of the Koszul bicomplex $\KK_\sspot(I)$ are
obtained from the chain complexes of Koszul simplicial complexes
of~$I$ by tensoring with~$\kk[\xx]$ over~$\kk$ (use
Lemma~\ref{l:basis} if this is not clear from
Convention~\ref{conv:koszul} and Definition~\ref{d:KK}).  The given
splittings $\del^{\bb+}$ thus induce a vertical splitting~$d^+$
of~$\KK_\sspot(I)$.

The natural Wall complex (Definition~\ref{d:natural-wall}) of this
vertically split Koszul bicomplex minimally resolves~$I$ by
Corollary~\ref{c:wall-koszul}.  The differentials in this resolution
are, by Definition~\ref{d:wall}, $D = \sum_j\wt\omega_j$ for the
homomorphisms $\wt\omega_j = (\1 - d^+ d) (d_1 d^+)^{j-1} d_1 (\1
-\nolinebreak d d^+)$ from Remark~\ref{r:natural-wall}.  The goal is
to determine the action of~$d^+$ on $\NN^n$-degree~$\bb$ Koszul cycles
in $\wt Z_i K^\bb I \otimes \kk[\xx] \subseteq (\KK_\spot^\yy)_\bb
\otimes \kk[\xx]$, and then the action of~$d$ on the output of
this~$d^+$, and the action of~$d_1$ on the output of
this~$d$,~and~so~on.

The reason why this requires care is that the action of~$d^+$ on an
$\NN^n$-degree~$\bb$ element of $\KK_\sspot(I)$ depends on how the
element decomposes in the basis from Lemma~\ref{l:basis}: the vertical
splitting is~$\del^{\bb+}$ only on basis vectors the form $\zz^\tau
\otimes \yy^{\bb-\tau} \otimes \xx^\aa$.  It is therefore crucial that
the $\NN^n$-degree~$\bb$ elements in $(\KK_\spot^\yy)_\bb \otimes
\kk[\xx]$ all have the form $\zz^\tau \otimes \yy^{\bb-\tau} \otimes
1$ and are not (say) mixtures in which the $\xx$-factors have nonzero
$\NN^n$-degree.

In contrast to~$d^+$, the actions of $d_1$ and~$d$ do not depend on
the tensor decomposition.  Let us start with~$d$.  The isomorphism
$\bigl(\KK_\spot^\yy(I^\yy)\bigr){}_\bb \cong \wt C_\spot K^\bb I$ of
the $\NN^n$-graded components of the columns of~$\KK_\sspot(I)$ with
chain complexes of Koszul simplicial complexes identifies $\zz^\tau
\otimes \yy^{\bb-\tau}$ with the face~$\tau$.  As such,
Lemma~\ref{l:basis} identifies~$d$ with the simplicial boundary
operator~$\del^\bb$ of~$\wt C_\spot K^\bb I$.  This is true regardless
of the $\xx$-factor and, indeed, regardless of the $\NN^n$-degree,
although of course for different $\NN^n$-degrees the differential
occurs in a different Koszul simplicial complex.  Importantly, the
$\zz\yy$-degree, for purposes of the splitting~$d^+$, does not change
under~$d$, as is visible from Lemma~\ref{l:basis}.

Similarly, $d_1$ acts on simplices~$\tau = \zz^\tau \otimes
\yy^{\cc-\tau} \otimes \xx^{\bb-\cc}$ of~$K^\cc I$ (thought of as
residing in $\NN^n$-degree~$\bb$ of $\wt C_\spot K^\cc I \otimes
\kk[\xx]$) as the boundary operator, but in this case the
$\zz\yy$-degree of the boundary face $\tau - \ee_k$ has
$\zz\yy$-degree $\cc-\ee_k$.  Therefore $d_1 = d_1^{\ee_1} + \dots +
d_1^{\ee_n}$ decomposes into the components that alter the
$\zz\yy$-degree by $\ee_1,\dots,\ee_n$.  Substituting this
decomposition of~$d_1$ back into the formula for~$\wt\omega_j$ and
introducing the $\NN^n$-degree indices
$\bb,\bb_1,\bb_2,\dots,\bb_{\ell-1},\bb_\ell=\aa$ on the upward and
downward differentials yields the sum over saturated decreasing
lattice paths $\lambda \in \Lambda(\aa,\bb)$, as desired.
\end{proof}

\begin{remark}\label{r:omega_j}
The proof of Theorem~\ref{t:choice-of-splitting} shows that the only
summand $\wt\omega_j$ contributing to the component $\HH_{i-1} K^\aa I
\otimes\kk[\xx] (-\aa) \from \HH_i K^\bb I \otimes \kk[\xx] (-\bb)$
induced by the homomorphism $\HH_{i-1} K^\aa I \smash{\stackrel{\
D}\ffrom} \HH_i K^\bb I$ is~$\wt\omega_\ell$, where $\ell = |\bb| -
|\aa|$.
\end{remark}


%
\begin{excise}{%
  \begin{lemma}\label{l:weights}
  The weights in Definition~\ref{d:weight} are integers.
  \end{lemma}
  \begin{proof}
  The coefficients $c$ need not be integers, but each of their
  denominators divides the relevant~$\theta$ (let alone $\theta^2$)
  because these denominators are orders of torsion elements in the
  relevant homology groups by Lemma\verb=~\ref{l:integer-circuit}=,
  Lemma\verb=~\ref{l:integer-shrub}=, and Proposition\verb=~\ref{p:Y}=.
  \end{proof}
}\end{excise}%

\begin{proof}[Proof of Theorem~\ref{t:sylvan}]
Given Theorem~\ref{t:choice-of-splitting}, it remains only to prove
that the formula for~$D$ in Theorem~\ref{t:choice-of-splitting}
specializes to the canonical sylvan homomorphism with entries
$$%
  D_{\sigma\tau}
  =
  \sum_{\lambda \in \Lambda(\aa,\bb)}
  \frac{1}{\Delta_{i,\lambda} I}
  \sum_{\phi \in \Phi_{\sigma\tau}(\lambda)}
  w_\phi
$$
from Definition~\ref{d:sylvan-matrix} when all of the
splittings~$\del^{\bb_j+}$ are the Moore--Penrose pseudoinverses of
the differentials~$\del^{\bb_j}$.  The proof is lattice path by
lattice path, so fix henceforth a saturated decreasing lattice
path~$\lambda$ from~$\bb$ to~$\aa$ of length $\ell = |\bb| - |\aa|$.
Fix as well the simplices $\tau \in K^\bb_i I$ and $\sigma \in
K^\aa_{i-1} I$.

The Projection Hedge Formula (Proposition~\ref{p:projection-hedge})
for $\pi_{B_i} = \del_{i+1} \del^+_{i+1}$ shows that%
\enlargethispage*{1ex}%
\begin{align*}
\1 - \del_{i+1} \del^+_{i+1}
  &=
  \frac{1\,}{\Delta_i^S\!}
  \Bigl(\Delta_i^S \1 - \sum_{S_i} \delta_{S_i}^2 \beta_{S_i}\Bigr)
\\[-.2ex]
  &=
  \frac{1\,}{\Delta_i^S\!}
  \sum_{S_i} \delta_{S_i}^2(\1 - \beta_{S_i}).
\end{align*}
In the image of~$\tau$ under this homomorphism at~$\bb \in \NN^n$, the
coefficient on~$\tau_0$ in the summand for the stake set~$S_i^\bb$ is
the weight of the boundary-link $\tau_0$\,---\,$\tau$ by
\mbox{Definition}~\ref{d:chain-linked}.\ref{i:boundary}.  Summing over
stake sets and dividing by $\Delta_i^S K^\bb I$ yields the $\tau_0
\tau$ matrix entry in~$\del_{i+1} \del^+_{i+1}$.

Let $j \geq 1$.  In the image of any $i$-simplex $\tau_{j-1}$ under
$d_1 = d_1^{\ee_1} + \dots + d_1^{\ee_n}$, where $d_1^{\ee_k}$ alters
the $\zz\yy$-degree by~$\ee_k$ (see the end of the proof of
Theorem~\ref{t:choice-of-splitting}), the coefficient on~$\sigma_j$
is~$0$ unless $\sigma_j = \tau_{j-1} - \lambda_j$ as in
Definition~\ref{d:fence}, in which case the coefficient output by
Theorem~\ref{t:choice-of-splitting} is a sign---the correct one for
the containment
\raisebox{-.5ex}[0pt][0pt]{$\sigma_j$}\,%
\raisebox{.4ex}{\tiny$\diagup$}\,%
\raisebox{1ex}[0pt][0pt]{$\tau_{j-1}$}
by~Definition~\ref{d:weight}.

In the image of any $(i-1)$-simplex~$\sigma_j$ under the
Moore--Penrose pseudoinverse~$\del_i{}^{\!\!\bb_j+\!}$ of the boundary
$\del_i{}^{\!\!\bb_j\!}$, the coefficient on~$\tau_j$ in the summand
indexed by the hedge~$\st_i{}^{\!\bb_j\!}$ in the Hedge Formula
(Corollary~\ref{c:hedge}) is the weight of the chain-link
\raisebox{1ex}[0pt][0pt]{$\tau_{j\!}$}\,\,%
\raisebox{.25ex}{\tiny$\diagdown$}\,%
\raisebox{-.75ex}[0pt][0pt]{$\sigma_j$}
by Proposition~\ref{p:chain-linked}.  Summing over hedges and dividing
by $\Delta_i^\st K^{\bb_j} I$ yields the $\sigma_j \tau_j$ matrix
entry in~$\del_i{}^{\!\!\bb_j+}$.

In the image of any $(i-1)$-simplex~$\sigma_\ell$ under orthogonal
projection $\1\! - \del_i{}^{\!\!\aa+} \del_i^\aa$ to the cycles
\raisebox{0pt}[0pt][0pt]{$\wt Z_{i-1} K^\aa I$}, the coefficient
on~$\sigma$ in the summand indexed by the shrubbery~$T_{i-1}^\aa$ in
the Projection Hedge Formula (Corollary~\ref{c:projection-hedge} and
Proposition~\ref{p:projection-hedge}) is the weight of the cycle-link
$\sigma$\,---\,$\sigma_\ell$
Definition~\ref{d:chain-linked}.\ref{i:cycle}.  Summing over
shrubberies and dividing by $\Delta_{i-1}^T K^\aa I$ yields
the~$\sigma_\ell \sigma$~\mbox{matrix}~entry~in~$\1\! -\nolinebreak
\del_i{}^{\!\!\aa+} \del_i^\aa$.

Summing the products of these three kinds of sums and the sign over
all chain-link fences from~$\tau$ to~$\sigma$ yields the matrix
entry~$D_{\sigma\tau}$ by matrix multiplication from elementary linear
algebra.  Definition~\ref{d:sylvan-matrix} expresses this product of
sums as a sum of~products.
\end{proof}

\section{Noncanonical sylvan resolutions}\label{s:noncanonical}

\noindent
The master formula for Wall resolutions from Koszul simplicial
splittings (Theorem~\ref{t:choice-of-splitting}) has the consequence
that once the canonicality requirement is dropped, our constructions
work universally, combinatorially, and minimally.  The format is
basically the same as Theorem~\ref{t:sylvan}, but there is no division
and the weights are simpler.

\begin{defn}\label{d:simple-weight}
Each chain-link fence edge has a \emph{simple weight} over~$\kk$:
\begin{itemize}\itemsep=1ex
\item%
the boundary-link $\tau_0$\,---\,$\tau$ has simple weight
$c_\tau(\tau_0,S{}_i^{\,\bb})$,
\item%
the chain-link
\raisebox{2ex}[0pt][0pt]{$\tau_j\!$}\,\,\raisebox{.5ex}{\tiny$\diagdown$}\,%
\raisebox{-1ex}[0pt][0pt]{$\sigma_j$}
has simple weight $c_{\sigma_j}(\tau_j,\st_i{}^{\!\!\bb_j})$,
\item%
the containment\,
\raisebox{-1.25ex}[0pt][0pt]{$\,\sigma_j$}\raisebox{.25ex}{\tiny$\diagup$}\,%
\raisebox{1.25ex}[0pt][0pt]{$\tau_{j-1}$} 
has simple weight $(-1)^{\sigma_j \subset \tau_{j-1}}$, and
\item%
the cycle-link $\sigma$\,---\,$\sigma_\ell$ has simple weight
$c_{\sigma_\ell}(\sigma,T_{i-1}^\aa)$.
\end{itemize}
The \emph{simple weight} of the fence $\phi$ is the
product~$w^\kk_\phi$ of the simple weights on its edges.
\end{defn}

\begin{defn}\label{d:community}
Fix a CW complex~$K$.  A \emph{community} in~$K$ is a sequence
$$%
  \st_\spot
  =
  (\st_0,\st_1,\st_2,\dots)
  \quad\text{with}\quad
  T_i \cap S_i = \nothing
  \text{ for all }i.
$$
\end{defn}

\begin{prop}\label{p:community}
Any community $\st_\spot$ induces a differential
$\del^+_{\st_\subspot}\!$ over~$\kk$ such that
\begin{enumerate}
\item\label{i:dd+d}%
$\del_i \del^+_{\st_i} \del_i = \del_i$ and
\item\label{i:d+dd+}%
$\del^+_{\st_i} \del_i \del^+_{\st_i} = \del^+_{\st_i}$.%
\end{enumerate}
\end{prop}
\begin{proof}
The disjointness of $T_{i-1}$ and~$S_{i-1}$ means that $T_{i-1}
\subseteq \ol S_{i-1}$, so $\del^+_{\st_{i+1}}\del^+_{\st_i} = 0$ by
Definition~\ref{d:hedge-splitting}.  Property~\ref{i:dd+d} follows
from Proposition~\ref{p:hedge-splitting} because $\zeta_{T_i}(\tau)$
is a cycle---that is, $\del_i\zeta_{T_i}(\tau) = 0$.
Property~\ref{i:d+dd+} is immediate from
Definition~\ref{d:hedge-splitting}, with both sides of the equation
being~$0$ for non-stakes and $\tau$ when applied to any
boundary~$\del_i\tau$.
\end{proof}


\begin{defn}\label{d:noncanonical-sylvan-matrix}
Fix a monomial ideal~$I$ and a community (Definition~\ref{d:community}
and Proposition~\ref{p:community}) for each Koszul simplicial
complex~$K^\bb I$.  These data endow each lattice path $\lambda \in
\Lambda(\aa,\bb)$ with a fixed hedgerow~$ST_i^\lambda$ for $i =
0,\dots,n$.  The \emph{sylvan homomorphism}
$$%
  \wt C_{i-1} K^\aa I
  \ \stackrel{\ D\ =\ D^{\aa\bb}}\filleftmap\
  \wt C_i K^\bb I
$$
for these data is given by its \emph{sylvan matrix}, whose entry
$D_{\sigma\tau}$ for $\tau \in K_i^\bb I$ and $\sigma \in K_{i-1}^\aa
I$ is the sum, over all lattice paths from~$\bb$~to~$\aa$, of the
weights of all chain-link fences from~$\tau$ to~$\sigma$ that are
subordinate (Remark~\ref{r:subordinate}) to the relevant
hedgerow~$ST_i^\lambda$:
$$%
  D_{\sigma\tau}
  =
  \sum_{\lambda \in \Lambda(\aa,\bb)}
  \sum_{\substack{\phi \in \Phi_{\sigma\tau}(\lambda)\\\phi\vdash ST_i^\lambda}}
  w^\kk_\phi.
$$
\end{defn}

\begin{cor}\label{c:noncanonical-sylvan}
Fix a monomial ideal~$I$ and a community for each Koszul simplicial
complex~$K^\bb I$.  The sylvan homomorphism for these data on each
comparable pair $\bb \succ \aa$ of lattice points induces a
homomorphism $\wt Z_{i-1} K^\aa I \from \wt Z_i K^\bb I$ that vanishes
on~$\wt B_i K^\bb I$, and hence it induces a well defined \emph{sylvan
homology morphism} $\HH_{i-1} K^\aa I \from \HH_i K^\bb I$.  The
induced homomorphisms
$$%
  \HH_{i-1} K^\aa I \otimes \kk[\xx] (-\aa)
  \,\from\,
  \HH_i K^\bb I \otimes \kk[\xx] (-\bb)
$$
of\hspace{.3ex} $\NN^n$-graded free $\kk[\xx]$-modules constitute a
minimal free resolution of~$I$.
\end{cor}
\begin{proof}
The proof of Theorem~\ref{t:sylvan} in Section~\ref{s:resolutions},
which is already done lattice path by lattice path, works mutatis
mutandis in this setting but simplifies because the fixed hedgerows
eliminate the summations over stake sets, hedges, and shrubberies.
\end{proof}

\begin{remark}\label{r:dependencies}
Since Corollary~\ref{c:noncanonical-sylvan} occurs at the end of this
paper, it is worth taking precise account of the relatively meager
prerequisites---beyond standard constructions like $K^\bb I$---on
which its statement (but not its proof) relies.  It requires the
notions of
\begin{itemize}
\item%
shrubbery, stake, hedge (Example~\ref{e:CW}) and their coefficients
(Definition~\ref{d:chain-linked});
\item%
hedgerow (Definition~\ref{d:hedgerow}) to assemble this combinatorics
along lattice~paths;
\item%
chain-link fence (Definition~\ref{d:fence}) with simple weights
(Definition~\ref{d:simple-weight}); and
\item%
community \hspace{-.15ex}(Definition\,\ref{d:community}) and hedge
splitting \hspace{-.15ex}(Definition\,\ref{d:hedge-splitting}) for
the~\mbox{differential}.
\end{itemize}
\end{remark}

\begin{remark}\label{r:sylvan-resolution}
In general, a minimal free resolution
ought to be called \emph{sylvan} if its differentials are expressed as
linear combinations of those for individual choices of hedges.  Thus
the canonical resolutions in Theorem~\ref{t:sylvan} are sylvan because
its differentials are weighted averages of differentials from hedges,
and the resolutions in Corollary~\ref{c:noncanonical-sylvan} are
sylvan because each fixes single choices of hedges.  Of course, all
minimal free resolutions of a given graded ideal are isomorphic; the
question is how the resolution is expressed.  Usually in commutative
algebra the differentials are expressed by selecting bases for the
syzygies.  In contrast the sylvan method avoids choosing such bases,
even in Corollary~\ref{c:noncanonical-sylvan}, because the syzygies
are naturally homology vector spaces.  Instead, the sylvan method
selects bases for chains in a manner that descends to homology.
\end{remark}

\begin{remark}\label{r:computation}
Canonical sylvan resolutions in Theorem~\ref{t:sylvan} are not suited
to efficient algorithms, as they require storage, manipulation, and
sums over bases for chains in simplicial complexes.  In contrast,
noncanonical sylvan resolutions could potentially lead to efficient
algorithmic computation of free resolutions, since they select bases
not for chains but for homology and cohomology (see
Remark~\ref{r:coboundary}) in each $\NN^n$-degree while never actually
computing homology.  It helps that communities have been independently
and simultaneously invented in the context of persistent homology,
where there are called ``tripartitions''
\cite{edelsbrunner-olsbock2018}.  The computational algebra of these
could be particularly helpful, as relevant algorithms have been
implemented.
\end{remark}

\begin{example}\label{e:RP2}
Noncanonical sylvan resolutions provide combinatorial minimal free
resolutions of monomial ideals whose Betti numbers vary with the
characteristic of the field, such as the Stanley--Reisner ideal of the
six-vertex triangulation of the real projective plane.  In any
characteristic other than~$2$, this ideal has a minimal cellular free
resolution of length~$2$; see \cite[Section~4.3.5]{cca}, for instance.
But in characteristic~$2$, the top Betti number is at homological
stage~$3$, namely $\betti 3\1 I = 1$, where here $\1 = (1,1,1,1,1,1)$.
A~sylvan resolution compensates by selecting hedges that respect the
dependencies in characteristic~$2$.  In particular, although the stake
set $S_2^\1 = \nothing$ at~$\1$ is forced, because $K^\1 I = \RR\PP^2$
has dimension~$2$, the stake sets at degrees $\1 - \ee_i$ that differ
from~$\1$ by a standard basis vector~$\ee_i$ have cardinality~$1$ in
characteristic~$2$, each consisting of any single edge in the relevant
Koszul simplicial complex, thereby allowing the construction of
chain-link fences to get started.
\end{example}

\begin{remark}\label{r:eliahou-kervaire}
Judicious choices of communities in
Corollary~\ref{c:noncanonical-sylvan} can recover known special
classes of resolutions of monomial ideals, such as the
Eliahou--Kervaire resolution of any Borel-fixed or stable ideal
\cite{eliahou-kervaire1990} (see also \cite[Chapter~2]{cca}) or any
planar graph resolution of a trivariate ideal \cite{miller-Planar2002}
(see also \cite[Chapter~3]{cca}).  These assertions require proof;
they are planned for subsequent papers.
\end{remark}



\begin{thebibliography}{CLR$^{+\!}$15}
\raggedbottom

\bibitem[BW02]{batzies-welker2002}
Ekkehard Batzies and Volkmar Welker, \emph{Discrete Morse theory for
  cellular resolutions}, J.\,Reine Angew. Math. \textbf{543} (2002),
  147-–168.

\bibitem[BPS98]{bayer-peeva-sturmfels1998}
Dave Bayer, Irena Peeva, and Bernd Sturmfels, \emph{Monomial
  resolutions}, Math. Res. Lett. \textbf{5} (1998), no.~1--2, 31--46.

\bibitem[BS98]{bayer-sturmfels1998}
Dave Bayer and Bernd Sturmfels, \emph{Cellular resolutions of monomial
  modules}, J.\,Reine Angew. Math. \textbf{502} (1998), 123--140.

\bibitem[BT90]{benTal-teboulle1990}
Aharon Ben-Tal and Marc Teboulle,
  \emph{A geometric property of the least squares solution of linear
  equations}, Linear Algebra Appl.\,\textbf{139} (1990), 165--170.

\bibitem[Ber86]{berg1986}
Lothar Berg, \emph{Three results in connection with inverse matrices},
  Proceedings of the symposium on operator theory (Athens,
  1985), Linear Algebra Appl.\,\textbf{84} (1986), 63--77.

\bibitem[BH98]{bruns-herzog}
Winfried Bruns and J\"urgen Herzog, \emph{Cohen--Macaulay rings},
  revised edition, Cambridge Studies in Advanced Mathematics Vol.~39,
  Cambridge University Press, Cambridge, 1998.

\bibitem[CCK15]{catanzaro-chernyak-klein2015}
Michael J. Catanzaro, Vladimir Y. Chernyak, and John R. Klein,
  \emph{Kirchhoff's theorems in higher dimensions and Reidemeister
  torsion}, Homology Homotopy Appl.~\textbf{17} (2015), no.\,1, 165--189.

\bibitem[CCK17]{catanzaro-chernyak-klein2017}
Michael J. Catanzaro, Vladimir Y. Chernyak, and John R. Klein,
  \emph{A higher Boltzmann distribution},
  J. Appl. Comput. Topol. \textbf{1} (2017), no.\,2, 215--240.

\bibitem[CT19]{clark-tchernev2019}
Timothy B. P. Clark and Alexandre B. Tchernev, \emph{Minimal free
  resolutions of monomial ideals and of toric rings are supported on
  posets}, Trans. Amer. Math. Soc. \textbf{371} (2019), no.\,6,
  3995--4027.


\bibitem[DKM09]{duval-klivans-martin2009}
Art M. Duval, Caroline J. Klivans, and Jeremy L. Martin,
  \emph{Simplicial matrix-tree theorems},
  Trans. Amer. Math. Soc.~\textbf{361} (2009), 607-611.

\bibitem[DKM11]{duval-klivans-martin2011}
Art M. Duval, Caroline J. Klivans, and Jeremy L. Martin,
  \emph{Cellular spanning trees and Laplacians of cubical complexes},
  Adv. in Appl. Math.~\textbf{46} (2011), 247-–274.

\bibitem[DKM09]{duval-klivans-martin2015}
Art M. Duval, Caroline J. Klivans, and Jeremy L. Martin,
  \emph{Cuts and flows of cell complexes}, J. Algebraic
  Combin.~\textbf{41} (2015), no.\,4, 969--999.

\bibitem[Eag90]{eagon1990}
John~A. Eagon, \emph{Partially split double complexes with an
  associated Wall complex and applications to ideals generated by
  monomials}, J. Algebra \textbf{135} (1990), no.~2, 344--362.

\bibitem[E\"O18]{edelsbrunner-olsbock2018}
Herbert Edelsbrunner and Katharina \"Olsb\"ock, \emph{Holes and
  dependences in an ordered complex}, Computer Aided Geometric Design
  \textbf{73} (2019), 1--15.

\bibitem[EFS03]{eisenbud-floystad-schreyer2003}
David Eisenbud, Gunnar Fl\o ystad, and Frank-Olaf Schreyer,
  \emph{Sheaf cohomology and free resolutions over exterior algebras},
  Transactions of the American Mathematical Society \textbf{355}
  (2003), no.\,11, 4397--4426.

\bibitem[EK90]{eliahou-kervaire1990}
Shalom Eliahou and Michel Kervaire, \emph{Minimal resolutions of some
  monomial ideals}, J.\,Algebra \textbf{129} (1990), no.~1, 1--25.

\bibitem[EN62]{eagon-northcott1962}
John A. Eagon and Douglas Geoffrey Northcott, \emph{Ideals defined by
  matrices and a certain complex associated with them},
  Proc. Roy. Soc. Ser. A \textbf{269} (1962) 188--204.


\bibitem[Hoc77]{hochster1977}
Melvin Hochster, \emph{Cohen--Macaulay rings, combinatorics, and
  simplicial complexes}, Ring theory, II (Proc. Second Conf.,
  Univ. Oklahoma, Norman, Okla., 1975) (B.~R. McDonald and R.~Morris,
  eds.), Lecture Notes in Pure and Applied Mathematics Vol.~26,
  Marcel Dekker, New York, 1977, pp.~171--223.

\bibitem[Kal83]{kalai1983}
Gil Kalai, \emph{Enumeration of $\mathbb{Q}$-acyclic simplicial
  complexes}, Israel J. Math.~\textbf{45}, no.\,4 (1983), 337-–351.

\bibitem[Lyo09]{Lyons2009}
Russell Lyons, \emph{Random complexes and $\ell_2$-Betti numbers},
  J. Topol. Anal.~1 (2009), 153--175.

\bibitem[Lyu88]{lyubeznik1988}
Gennady Lyubeznik, \emph{A new explicit finite free resolution of
  ideals generated by monomials in an $R$-sequence}, J. Pure
  Appl. Algebra \textbf{51} (1988), no.\,1--2, 193--195.

\bibitem[Mil00]{alexdual}
Ezra Miller, \emph{The Alexander duality functors and local duality
  with monomial support}, Journal of Algebra \textbf{231} (2000),
  180--234.


\bibitem[Mil02]{miller-Planar2002}
Ezra Miller, \emph{Planar graphs as minimal resolutions of trivariate
  monomial ideals}, Documenta Math.~\textbf{7}~(2002), 43--90.

\bibitem[MS05]{cca}
Ezra Miller and Bernd Sturmfels, \emph{Combinatorial commutative
  algebra}, Graduate Texts in Mathematics, vol.~227, Springer-Verlag,
  New York, 2005.

\bibitem[MSY00]{miller-sturmfels-yanagawa2000}
Ezra Miller, Bernd Sturmfels, and Kohji Yanagawa, \emph{Generic and
  cogeneric monomial ideals}, J. Symbolic Comput. \textbf{29} (2000),
  691-- 708.


\bibitem[OW16]{olteanu-welker2016}
Anda Olteanu and Volkmar Welker, \emph{The Buchberger resolution},
  J. Commut. Algebra \textbf{8} (2016), no.\,4, 571--587.

\bibitem[Pet09]{petersson2009}
Anna Petersson, \emph{Enumeration of spanning trees in simplicial
  complexes}, Uppsala University Department of Mathematics,
  Report~2009:13 (May 18, 2009).

\bibitem[Tay66]{taylor1966}
Diana Taylor, \emph{Ideals generated by monomials in an
  {$R$}-sequence}, Ph.D.~thesis, University of Chicago, 1966.

\bibitem[Tch19]{tchernev2019}
Alexandre Tchernev, \emph{Dynamical systems on chain complexes and
  canonical minimal resolutions}, preprint, 2019.
  \textsf{arXiv:math.AC/1909.08577}

\bibitem[TV15]{tchernev-varisco2015}
Alexandre Tchernev and Marco Varisco, \emph{Modules over categories
  and Betti posets of monomial ideals},
  Proc. Amer. Math. Soc. \textbf{143} (2015), no.\,12, 5113--5128.

\bibitem[Yuz99]{yuzvinsky1999}
Sergey Yuzvinsky, \emph{Taylor and minimal resolutions of homogeneous
  polynomial ideals}, Math. Res. Lett.\,\textbf{6} (1999), no.\,5--6,
  779--793.

\end{thebibliography}
\end{document}